\documentclass{amsart}
\usepackage[utf8]{inputenc}
\usepackage[algonl,boxed,norelsize,lined]{algorithm2e}
\usepackage{amscd,amsmath,amsthm,amssymb,amsfonts}
\usepackage{fullpage}
\usepackage{comment}
\usepackage[all]{xy}
\usepackage{datetime}
\usepackage{color}
\usepackage[english]{babel}
\usepackage[T1]{fontenc}
\usepackage{hyperref}
\usepackage{enumitem}
\usepackage{theoremref}
\usepackage{wrapfig}
\usepackage{booktabs}
\usepackage{youngtab}
\usepackage[top=3cm,bottom=2cm,left=3cm,right=3cm,marginparwidth=1.75cm]{geometry}

\usepackage{amsmath}
\usepackage{amssymb}
\usepackage{amsthm}
\usepackage{verbatim}

\usepackage{tikz}
\usetikzlibrary{decorations.markings}
\usetikzlibrary{calc}

\newtheorem{Theorem}{Theorem}[section]
\newtheorem{algo}[Theorem]{Algorithm}
\newtheorem{Lemma}[Theorem]{Lemma}
\newtheorem{Corollary}[Theorem]{Corollary}
\newtheorem{Proposition}[Theorem]{Proposition}

\newtheorem{Definition}[Theorem]{Definition}
\newtheorem{Example}[Theorem]{Example}

\DeclareMathOperator{\init}{in}
\DeclareMathOperator{\Gr}{Gr}
\DeclareMathOperator{\GF}{GF}
\DeclareMathOperator{\GR}{GR}
\DeclareMathOperator{\gr}{gr}
\DeclareMathOperator{\trop}{Trop}
\DeclareMathOperator{\conv}{conv}
\DeclareMathOperator{\supp}{supp}
\DeclareMathOperator{\cone}{cone}
\DeclareMathOperator{\spec}{Spec}
\DeclareMathOperator{\proj}{Proj}
\DeclareMathOperator{\val}{\mathfrak v}

\newcommand{\RR}{\mathbb{R}}
\newcommand{\BB}{\mathbb{B}}

\definecolor{caribbeangreen}{rgb}{0.0, 0.8, 0.6}
\definecolor{capri}{rgb}{0.0, 0.75, 1.0}

\title{A survey on toric degenerations of projective varieties}
\author{Lara Bossinger}
\address{
Instituto de Matem\'aticas Unidad Oaxaca, 
Universidad Nacional Aut\'onoma de M\'exico,
Le\'on 2, altos, 
Centro Hist\'orico,
68000 Oaxaca,
Mexico}
\email{lara@im.unam.mx}
\date{}

\begin{document}

\maketitle

\begin{abstract}
In this survey we summarize the constructions of toric degenerations obtained from valuations and Gröbner theory and describe in which sense they are equivalent. 
We show how adapted bases can be used to generalize the classical Newton polytope to what we call a $\mathbb B$-Newton polytope.
The $\mathbb B$-Newton polytope determines the Newton--Okounkov polytopes of all Khovanskii-finite valuations sharing the adapted standard monomial basis $\mathbb B$.
\end{abstract}

\section{Introduction}\label{sec:intro}

Toric varieties are popular objects in algebraic geometry due to a dictionary between their geometric properties (e.g. dimension, degree)  and properties of associated combinatorial objects (e.g. fans, polytopes). This dictionary can be extended from toric varieties to varieties admitting a \emph{toric degeneration}. A toric degeneration is a (flat) family of varieties that share many properties with each other (e.g. dimension, degree, Hilbert-polynomial). This family contains the variety we are interested in, a toric variety, so properties of our variety can be read from the combinatorics of the toric variety.

The study of toric degenerations has various applications in pure and applied mathematics, e.g. in probability, statistics, and mathematical biology.
Tailored to the variety of interest, it is a great challenge to decide which toric degeneration has the desired properties. The task is therefore to study and compare possible constructions.

\medskip
Besides its applications in classical algebraic geometry, toric degenerations have proven to be useful in several other subjects such as 
\begin{itemize}
    \item Symplectic geometry \cite{NNU2,Hamilton_Kaehler,FLP, HP, HK, Kaveh_toric}, 
    \item Newton--Okounkov bodies \cite{LM09,KK12,KueL},
    \item Representation theory \cite{GL96,Cal02,AB04, KM05,Hoew_Branching, FFL_PBWtypeA,Kaveh_crystal, FFL_birat},
    \item Mirror symmetry \cite{Giv_Toda,BCFvSK,Batyrev_toric,GrossSiebert_toric,FOOO_toric, ACGK_mutation, Nishinou_Gromov-Witten},
    \item Cluster algebras \cite{GHKK14,BFFHL,RW17,BFMN},
    \item Numerical and computational algebraic geometry \cite{Collart_toricDeg,BLMM,BSW_numerical}
    \item Algebraic statistics \cite{KMS_quasisymmetry,Daniel_trees}
\end{itemize}
The above list and the included citations are far from being complete as the subject is broad and new applications are discovered on a regular basis. I apologize if I have missed your favorite paper using toric degenerations and I would be happy to receive emails with hints to more exiting applications.

The aim of this survey is to describe two main constructions of toric degenerations and how they are related. 
In particular, we focus on the constructions from valuations which go back to Anderson \cite{An13} and those from Gröbner theory or the tropicalization of an ideal.
In practice the \emph{bridges} connecting one construction to another are particularly useful as each approach has its own benefits and shortcomings.

\medskip
To be more precise consider a projective variety $X$. A {\bf toric degeneration of $X$} is a flat morphism $\phi:\mathfrak X\to \mathbb A^1$ that trivializes away from the fibre over $0\in \mathbb A^1$
\[
\xymatrix{
\mathfrak X\setminus \phi^{-1}(0)\ar[rr]^\sim  \ar[dr]_\phi& &  X\times \mathbb A^1\setminus\{0\}\ar[dl]\\
& \mathbb A^1\setminus\{0\}
}
\]
We will work with $T$-equivariant toric degenerations, that is we assume that the action on the central fibre is an extension of the torus action on $X$. 

\medskip
\noindent
{\bf Outline:} We summarize the background on valuations in \S\ref{sec:val} and on Gröbner theory and tropicalization of ideals in \S\ref{sec:initial ideals and trop}. In \S\ref{sec:val and trop} we explain the equivalence of the constructions of toric degenerations from valuations and from the tropicalization of an ideal \cite{KM16,Bos20_fullrank}. In \S\ref{sec:gröbner?} we consider more general \emph{algebraic toric degenerations}
and under which circumstances they can be obtained as the toric degenerations from a valuation \cite{KMM_dgeneration-cx1}.
In \S\ref{sec:adapted} we explain the importance of adapted bases for toric degenerations. In particular, in \S\ref{sec:polytopes from adapted} we show how adapted bases give rise to $\mathbb B$-Newton polytopes that project to all Newton--Okounkov polytopes of valuations that share an adapted basis.
In \S\ref{sec:wall-crossing} we recall the definition of wall-crossing formulas for Newton--Okounkov polytopes \cite{EH-NObodies}. We review this notion from a more geometric point of view in the context of flat families incorporating various toric degenerations in \S\ref{sec:families of gröbner}, \cite{BMN}. 
In \S\ref{sec:Gr36} we elaborate on the example of the Grassmannian $\Gr_3(\mathbb C^6)$.

\medskip
\noindent
{\bf Acknowledgements:} 
I would like to thank the editors of the proceeding for the opportunity to contribute with a survey on my favorite topic.
The ideas of \S4.1 arose in the context of Newton--Okounkov bodies for cluster varieties in collaboration with Man-Wai Cheung, Timothy Magee and Alfredo Nájera Chávez. I am grateful for all the inspiring discussions over the course of our collaboration.

\section{Preliminaries}\label{sec:prelim}
In this section we introduce the main tools used in this article to construct toric degenerations. In particular we review background on valuations and Newton--Okounkov bodies as well as background on initial ideals, Gröbner theory and the tropicalization of an ideal.
We use the maximum convention throughout the paper which might imply slight differences in the definitions (mostly just a sign) in comparison to the original articles cited.

\subsection{Valuations}\label{sec:val}

Let $k$ be an algebraically closed field of characteristic zero.
Throughout the paper we denote by $A$ a {\bf positively (multi-)graded algebra and domain}, that is $A=\bigoplus_{w\in \mathbb Z^m_{\ge 0}}A_w$.
Let $\Gamma$ be an abelian group that is totally ordered by $<$.
By a {\bf (Krull) valuation} on $A$ we mean a map $\nu:A\setminus \{0\}\to \Gamma$ that satisfies for all $a,b\in A$ and $c\in k$
\[
\nu(ab)=\nu(a)+\nu(b), \quad \nu(a+b)\le \max\{\nu(a),\nu(b)\},\quad \nu(ca)=\nu(a).
\]
If $\nu$ only satisfies $\nu(ab)\le \nu(a)+\nu(b)$ it is called a {\bf quasivaluation}.
Notice that the image of a valuation $\nu$ carries the structure of an additive semigroup. It is therefore called the {\bf value semigroup} of $\nu$ and we denote it by $S(A,\nu)$.
The {\bf rank} of $\nu$ is defined as the rank of the group completion of its semigroup inside $\Gamma$, $\nu$ is said to have {\bf full rank} if its rank coincides with the Krull dimension of $A$.
Every valuation induces a filtration on $A$ with filtered pieces for $\gamma \in \Gamma$ defined by
\[
\mathcal F_{\nu,\gamma}:=\{a\in A:\nu(a)\ge \gamma \}\quad  \left(\text{resp. }\mathcal F_{\nu,>\gamma}:=\{a\in A:\nu(a)> \gamma\}\right).
\]
The {\bf associated graded algebra} is $\gr_\nu(A):=\bigoplus_{\gamma\in \Gamma} \mathcal F_{\nu,\gamma}/\mathcal F_{\nu,>\gamma}$.
There is a natural quotient map of vector spaces from $A$ to $\gr_\nu(A)$ given by sending $f\in A$ to $\mathcal F_{\nu,\nu(f)}/\mathcal F_{\nu,>\nu(f)}$, denote its image by $\hat f\in \gr_\nu(A)$. Note that $\nu(fg)=\nu(f)+\nu(g)$ implies that $\widehat{fg}=\hat{f}\hat{g}$.
If the quotients $F_{\nu,\gamma}/\mathcal F_{\nu,>\gamma}$ are at most one-dimensional, then we say $\nu$ has {\bf one-dimensional leaves}.
This property is desirable as it gives an identification
\[
\gr_\nu(A)\to k[S(A,\nu)], \quad \text{given by} \quad \hat f_\gamma \mapsto \nu(f_\gamma),
\]
where $\hat f_\gamma\in \mathcal F_{\nu,\gamma}/\mathcal F_{\nu,>\gamma}$ is a generator and $f_\gamma\in A$ lies in the preimage of $\hat f_\gamma$ under the quotient map $\hat{}:A\to \gr_\nu(A)$.
It is a consequence of Abhyankar's inequality that full-rank valuations have one-dimensional leaves.

An important definition is the notion of a {\bf Khovanskii basis} for a valuation $\nu$: that is a subset $B$ of $A$ whose image in $\gr_\nu(A)$ is an algebra generating set. 
It is not hard to see that if $B$ is a Khovanskii basis for $\nu$ then the set $\{\nu(b):b\in B\}$ generates the value semigroup \cite[Lemma 2.10]{KM16}.

A valuation is called {\bf homogeneous} if it respects the grading on $A$, more precisely if $f\in A$ has homogeneous presentation $\sum_{i} f_i$ then $\nu(f)=\max\{\nu(f_i)\}$.
A valuation is {\bf fully homogeneous} if $\nu(f)=(\deg(f),\nu'(f))$, that is $S(A,\nu)\subset \mathbb Z_{\ge 0}^m\times \Gamma'$.
Any homogeneous valuation is obtained from a fully homogeneous one by composing with an isomorphism of semigroups \cite[Remark 2.6]{IW20}. So when studying homogeneous valuation we may without loss of generality assume they are fully homogeneous.

Given a fully homogeneous valuation $\nu:A\setminus \{0\}\to \mathbb Z^m_{\ge 0}\times \Gamma'$ we define its {\bf Newton--Okounkov cone}
\begin{equation}\label{eq:NOcone}
C(A,\nu):=\overline{\cone(S(A,\nu))}=\overline{\cone(\nu(f):f\in A)},    
\end{equation}
where the closure (in the Euclidean topology) is taken inside $(\mathbb Z^m_{\ge 0}\times \Gamma')\otimes_{\mathbb Z} \mathbb R$.
Let $\Gamma'_{\mathbb R}=\Gamma'\otimes_{\mathbb Z}\mathbb R$.
The {\bf Newton--Okounkov body} of $\nu$ is then defined as the intersection
\begin{equation}\label{eq:NObody}
\Delta(A,\nu):=C(A,\nu) \cap \{(1,\dots,1)\}\times \Gamma'_{\mathbb R},    
\end{equation}
where $(1,\dots,1)$ denotes the element whose entries are all one in $\mathbb Z^m_{\ge 0}$.
The definition was introduced independently by Lazarsfeld--Mustata \cite{LM09} and Kaveh--Khovanskii \cite{KK12} who based their work on a construction of Okounkov \cite{Oko98}.
Newton--Okounkov bodies far generalize Newton polytopes of polynomials and carry a lot of information about the algebra $A$ or the (weighted)projective variety $X=\proj(A)$\footnote{Recall, that the projective spectrum of the $\mathbb Z^m_{\ge 0}$-graded polynomial ring whose generators have degrees $d_ie_i$, $\{e_1,\dots,e_m\}$ being the standard basis of $\mathbb Z^m$, is the weighted projective space $\mathbb P(d_1,\dots,d_m)$. In particular, if $A$ is multigraded, $\proj(A)$ can be seen as a subvariety of a weighted projective space. For details we refer to \cite{Dolgachev_wtProjVar}.}. 

\begin{Theorem}[Corollary 3.2 \cite{KK12}]
Let $X=Proj(A)$ and $\nu:A\setminus \{0\}\to \mathbb Z^d$ be a full-rank homogeneous valuation. Then the dimension $q$ of $\Delta(A,\nu)$ coincides with the dimension of $X$ and moreover, the $q$-dimensional integral volume of $\Delta(A,\nu)$ multiplied by $q!/{\rm ind}(S(A,\nu))$ is the degree of $X$, where ${\rm ind}(S(A,\nu))$ refers to the index of the sublattice spanned by $S(A,\nu)$ inside $\mathbb Z^d$.
\end{Theorem}

In general the Newton--Okounkov body of a valuation need not be bounded nor polyhedral, but they are convex. Computing them is in general challenging, but much simplified when the valuation posses a finite Khovanskii basis as we will see in what follows.

\subsubsection{Khovanskii-finite valuations}
A {\bf Khovanskii-finite valuation} is a homogeneous (Krull) valuation of full rank whose value semigroup is finitely generated. In particular, Khovanskii-finite valuations have finite Khovanskii bases.
The concept was introduced and studied in great detail by Ilten and Wrobel in \cite{IW20}.

\medskip
The existence of Khovanskii bases has computational advantages.
Given a Khovanskii basis $\{b_1,\dots,b_n\}$ for $\nu$ we may represent $\gr_\nu(A)$ as a quotient of a polynomial ring $S:=k[x_1,\dots,x_n]$. Define
\[
\pi_\nu:k[x_1,\dots,x_n] \to \gr_\nu(A) \quad \text{by} \quad x_i\mapsto b_i.
\]
Then $I_\nu:=\ker(\pi_\nu)$ gives $S/I_\nu\cong \gr_\nu(A)$.

We say the value semigroup $S(A,\nu)\subset (\Gamma,<)$ is {\bf minimum well-ordered} if every subset of $S(A,\nu)$ has a unique minimal element with respect to $<$. In this case by \cite[Proposition 2.13]{KM16} the following version of the subduction algorithm terminates in finite time.

\begin{algo}
Let $A$ be positively graded algebra and domain, $\nu:A\setminus \{0\}\to (\Gamma,<)$ full-rank homogeneous Khovanskii-finite valuation with minimum well-ordered $S(A,\nu)$ with $\{b_1,\dots,b_n\}$ a Khovanskii basis.
\begin{itemize}
    \item[{\bf Input:}] $f\in A\setminus \{0\}$;
    \item[{\bf Output:}] a polynomial expression of $f$ in terms of $\{b_1,\dots,b_n\}$.
\end{itemize}
\begin{enumerate}
    \item As $\bar b_1,\dots, \bar b_n$ generate $\gr_\nu(A)$ we may find a polynomial expression for $\bar f$ in terms of $\bar b_1,\dots, \bar b_n$: $\bar f=p(\bar b_1,\dots,\bar b_n)$, here $p\in \pi_\nu^{-1}(\bar f)$.
    \item We distinguish two cases
    \begin{itemize}
        \item[a.] If $f=p(b_1,\dots,b_n)$ then {\bf output} $p$;
        \item[b.] If $\nu(f-p(b_1,\dots,b_n))<\nu(f)$ replace $f$ by $f-p(b_1,\dots,b_n)$ and go back to Step 1.
    \end{itemize} 
\end{enumerate}
\end{algo}

In particular, every Khovanskii basis is a generating set for the algebra.

\subsection{Initial ideals and tropicalization}\label{sec:initial ideals and trop}
Our second tool box for toric degenerations comes from Gröbner theory. For more detailed information we refer to \cite{HH_monomial,CLO_commut,Eisenbud,Stu96}.

\medskip
For $m\in \mathbb Z^n_{\ge 0}$ we write $x^m:=x_1^{m_1}\cdots x_n^{m_n}\in k[x_1,\dots,x_n]$.
A total order on the set of monomials in $S:=k[x_1,\dots,x_n]$ is a {\bf term order} if it satisfies:
\[
(i)\  1 < x^m\  \forall m\in \mathbb Z^n_{\ge 0} \setminus \{0\} \quad \text{and} \quad (ii)\  x^a<x^{b} \Rightarrow x^{a+c}<x^{b+c}\ \forall a,b,c\in \mathbb Z^n_{\ge 0}.
\]
The {\bf leading term} of an element $f=\sum c_ax^a\in S$ with respect to a term order $<$ is $\init_<(f)=c_bx^b:=\max_<\{c_ax^a:c_a\not =0\}$, where $c_b$ is called the {\bf leading coefficient} and $x^b$ is called the {\bf leading monomial}.
For an ideal in $I\subset S$ we define its {\bf initial ideal with respect to $<$} as 
\[
\init_<(I):=(\init_<(f):f\in I).
\]
The initial ideal is finitely generated and a generating set $G$ of $I$ that satisfies $(\init_<(g):g\in G)=\init_<(I)$ is called a {\bf Gr\"obner basis of $I$ with respect to $<$}.
Every ideal possesses only a finite number of distinct initial ideals \cite[Theorem 1.2]{Stu96}.
It has been shown by Mora and Robbiano that the initial ideals can be organized in a polyhedral fan \cite{MoraRobbiano}. 
To see how, we need the notion of initial ideals with respect to weight vectors:
fix $w\in \mathbb R^n$, we call it a {\bf weight vector} and define the {\bf initial form} of an element $f=\sum c_ax^a$ with respect to $w$ as 
\[
\init_w(f) =\sum_{b:\  w\cdot b=\max\{w\cdot a:c_a\not =0\}}c_bx^b.
\]
Notice that depending on $w$ and $f$ the initial form $\init_w(f)$ is not necessarily just one term.
Similarly, we define the {\bf initial ideal} of $I$ {\bf with respect to $w$} as $\init_w(I):=(\init_w(f):f\in I)$.
For any weight vector $w$ we may define the {\bf homogenization} of $I$ in $k[x_1,\dots,x_n,t]$: for a single element $f=\sum c_ax^a$ we set
\[
f^{h;w}:=\sum c_ax^a t^{\max\{w\cdot b:c_b\not =0\}-w\cdot a}.
\]
Similarly, for the ideal $I$ we define $I^{h;w}:=(f^{h;w}:f\in I)$. The homogenization of $I$ is a family of deformations of $I$ and the quotient algebra $A^{h;w}:=k[x_1,\dots,x_n,t]/I^{h;w}$ is a free $k[t]$-module \cite[\S15.8]{Eisenbud}.
Let $A^w:=S/\init_w(I)$.
The degeneration of $\spec(A)$ to $\spec(A^w)$ defined by $\spec(A^{h;w})$ is called a {\bf Gröbner degeneration}.

Given the ideal $I$ any term order can be {\bf represented} by a weight vector in $w\in \mathbb Z^n_{>0}$ (see, e.g. \cite[Lemma 3.1.1]{HH_monomial}), that is $\init_w(I)=\init_<(I)$.
Conversely, a weight vector $w$ belongs to the {\bf Gr\"obner region} $\GR(I)$ if there exists a term order $<$ such that $\init_<(\init_w(I))=\init_<(I)$. 
The Gr\"obner region carries a fan structure, called the {\bf Gröbner fan} $\GF(I)$ that was discovered by Mora and Robbiano in \cite{MoraRobbiano}. 
Two weight vectors $v,w\in \RR^n$ lie in the relative interior of a cone $C$, denoted by  $v,w\in C^\circ$, if and only if $\init_v(I)=\init_w(I)$.
The maximal dimensional cones in $\GF(I)$ correspond to monomial initial ideals associated with term orders on $S$.
These are particularly useful as they induce vector space bases for the quotient algebra $A=S/I$: we call a monomial $x^a$ that is {\em not} contained in $\init_<(I)$ a {\bf standard monomial}. The set $\BB_<:=\{\bar x^a\in A:x^a\not \in \init_<(I)\}$ is a vector space basis of $A$ called a {\bf standard monomial basis}.
In fact, if $w\in C$ for some maximal cone $C\in \GF(I)$ associated to $<$, then $\BB_<$ is a basis for the free $k[t]$-module $A^{h;w}$, see e.g. \cite[Proof of Theorem 15.17]{Eisenbud}.

The Gr\"obner fan has an interesting subfan that will lead us back to valuations on $A$: we define the {\bf tropicalization} of $I$:
\[
\trop(I):=\{w\in \GR(I): \init_w(I) \ \text{ does not contain monomials}\}.
\]
The dimension of $\trop(I)$ coincides with the Krulldimension of $A$ \cite{EinsiedlerKapronovLind_dimTrop}.
We may in fact reduce our attention to homogeneous ideals
according to \cite[Lemma 4]{ComputingTrop_BJSST}.
For homogeneous ideals we have $\GR(I)=\RR^n$ by \cite[Proposition 1.12]{Stu96}.
In this case the tropicalization is a pure fan whose dimension coincides with the Krull-dimension of $A$.
Moreover, $\trop(I)$ and $\GF(I)$ have a linear subspace $\mathcal L_I$, called the {\bf lineality space} which consists of elements $w\in \mathbb R^n$ such that $\init_w(I)=I$.
More precisely we have the following straight forward Lemma:

\begin{Lemma}\label{lem:grading and lineality}
Let $I$ be a (multi-)homogeneous ideal inside $S$ with respect to a $\mathbb Z^m_{\ge 0}$-grading given by  $\deg(x_{i_j})=e_i$ 
where $S=k[x_{i_j}:1\le i\le m,1\le j\le k_i]$ for some $k_i$ that satisfy $k_1+\dots+k_m=n$ and $\{e_i:1\le i\le m\}$ is the standard basis of $\mathbb Z^m$.
Then for $1\le i\le m$ we have 
\begin{equation}\label{eq:elements lineality}
    \ell_i:=(0,\dots,0,1,\dots,1,0\dots,0) \in \mathcal L_I
\end{equation}
where the $1$'s appear in the positions $i_1,\dots,i_{k_i}$.
\end{Lemma}

Among the maximal cones of $\trop(I)$ we may look for {\bf prime cones} whose associated initial ideal is binomial and prime. Hence, their vanishing sets define toric varieties. 
In particular, any Gr\"obner degeneration associated to a weight vector in the interior of a maximal prime cone is in fact a {\bf toric degeneration}.

\begin{Example}\label{example: curve}
Consider $I=(y^2z-x^3+z^3)\subset \mathbb C[x,y,z]$. The lineality space $\mathcal L_I$ is one dimensional generated by $(1,1,1)^T$.
We draw the Gröbner fan $\GF(I)$ modulo $\mathcal L_I$ inside the hyperplane $\{(w_1,w_2,w_3)^T\in \mathbb R^3:w_3=0\}$ in Figure~\ref{fig:example Gfan}.
The one skeleton whose maximal cones correspond to the rays in the above picture is the tropicalization of $I$.
\begin{figure}
    \centering
    \begin{tikzpicture}
\draw[dashed] (-2.5,0) -- (2,0);
\node[right] at (2,0) {\tiny $w_1$};
\draw[dashed] (0,2.5) -- (0,-2.5);
\node[above] at (0,2.5) {\tiny $w_2$};
\node[below] at (-2.5,2.5) {\small $w_3=0$};

\node at (4,0) {$\times \mathcal L_I=\left\langle \left(\begin{smallmatrix}1\\ 1\\ 1 \end{smallmatrix}\right)\right\rangle$};

\draw[fill,magenta, opacity=.3] (0,0) -- (1.5,2.25) to [out=320,in=10] (0,-2) -- (0,0);
\draw[fill,blue,opacity=.2] (0,0) -- (1.5,2.25) to [out=140,in=110] (-2,-1) -- (0,0);
\draw[fill, teal, opacity=.3] (0,0) -- (-2,-1) to [out=295,in=195] (0,-2) -- (0,0);
\draw[thick, blue,opacity=.2] (-2,-1) -- (0,0);
\draw[thick, teal] (0,0) -- (0,-2);
\draw[thick, magenta] (0,0) -- (1.5,2.25);

\node at (-1,1) {\small {$( y^2z )$}};
\node at (1,.5) {\small{$( x^3 )$}};
\node at (-1,-1) {\small{$( z^3 )$}};

\node[left, blue,opacity=.7] at (-2,-1) {\small{$(y^2z+ z^3)$}};
\node[below right, teal] at (0,-2) {\small{$(z^3-x^3) $}};
\node[above right, magenta] at (1.5,2.25) {\small{$ ( y^2z-x^3)$}};
\end{tikzpicture}   
    \caption{The Gröbner fan of $I=(y^2z-x^3+z^3)\subset \mathbb C[x,y,z]$ modulo $\mathcal L_I$, its one-skeleton is $\trop(I)$, and all initial ideals.}
    \label{fig:example Gfan}
\end{figure}
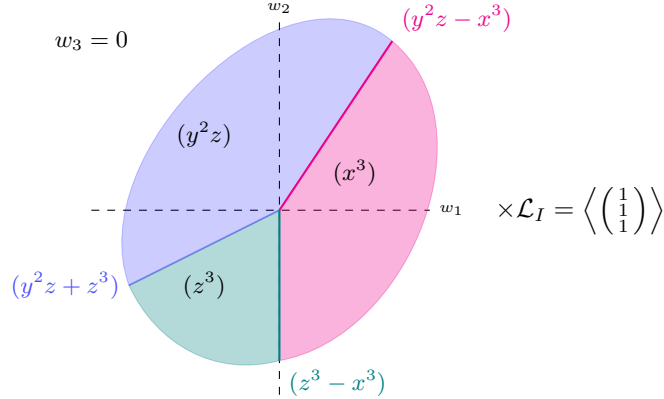

\end{Example}

\subsubsection{Initial ideals with respect to weighting matrices}
Before moving on to the next section we need a bit more background on a slight generalization of initial ideals: a \emph{higher dimensional analogue of Gr\"obner theory} (see e.g. \cite{FR_Hahn_higherRank}).

Recall that $S=k[x_1,\dots,x_n]$ and consider as before $f=\sum c_ax^a\in S$. 
We call a matrix $M\in \mathbb Q^{d\times n}$ a {\bf weighting matrix} and together with a linear order $\prec$ on $\mathbb Z^d$ we define the {\bf initial form} of $f$ with respect to $M$ as
\[
\init_M(f):=\init_{M,\prec}(f):=\sum_{b: \ Mb=\max_{\prec}\{Ma:c_a\not =0\}} c_bx^b.
\]
As before we define the {\bf initial ideal} of an ideal $I\subset S$ {\bf with respect to $M$ (and $\prec$)} as $\init_{M;\prec}(I):=(\init_M(f):f\in I)$.
To simplify notation we will drop the linear order from the index and simply assume that we have fixed it once and for all.
The Gr\"obner region also has a higher dimensional analogue: the {\bf $d^{\text{th}}$ Gr\"obner region} is denoted $\GR^d(I)$ and defined as the set of all weighting matrices $M\in \mathbb Q^{d\times n}$ such that there exists a term order $<$ on $S$ with $\init_<(I)=\init_<(\init_M(I))$.
Given $<$ let $C^d_<\subset \GR^d(I)$ be the set of all $M$ satisfying the previous relation.
We may also define equivalence classes of weighting matrices by setting $C_M:=\{M'\in \GR^d(I):\init_M(I)=\init_{M'}(I)\}$.
In the higher dimensional case several features of Gr\"obner theory still hold, among these the existence of standard monomial bases.
For example, $\GR^d(I)$ always contains the positive orthant $\mathbb Q_{\ge 0}^{d\times n}$ and if $I$ is homogeneous we have $\GR^d(I)=\mathbb Q^{d\times n}$ (see \cite[Lemma 8.7]{KM16} but be aware that the authors are using the minimum convention which introduces a sign).

We may use weighting matrices to define quasivaluations as follows.
Consider the quotient map $\pi:S\to S/I=:A$ and denote by $\bar f$ the coset of $f$ in the quotient. 
For $f=\sum c_a x^a\in S$ set $\tilde \nu_M(f):=\max_\prec\{Ma:c_a\not =0\}$.
This defines a valuation $\tilde \nu_M:S\setminus\{0\}\to \mathbb Z^d$.
By \cite[Lemma 3.2]{KM16} there exists a quasivaluation $\nu_M:A\setminus\{0\}\to (\mathbb Z^d,\prec)$ given for $\bar f\in A$ by
\[
\nu_M(\bar f)=\min_\prec\{\tilde \nu_M(f):f\in \bar f\}
\]
called the {\bf quasivaluation with weighting matrix $M$}. Its associated graded algebra, denoted $\gr_M(A)$, satisfies $\gr_M(A)\cong S/\init_M(I)$.
In particular, this isomorphism gives us standard monomial bases for $\gr_M(A)$: let $<$ be a term order with $M\in C^d_<$. Then $\mathbb B_<$ is a vector space basis for $\gr_M(A)$. 
Moreover, we may use $\mathbb B_<$ to compute the values of $\nu_M$: for $\bar f\in A$ let $\bar f=\sum_{\bar x^b\in \mathbb B_<} c_b\bar x^n$ be its expression in $\mathbb B_<$. Then
\[
\nu_M(\pi( f))=\max_{\prec}\{Mb:c_b\not =0\}.
\]
We explore standard monomial bases and their influence on valuations further in \S\ref{sec:adapted}.

\section{Valuations, tropicalization and toric degenerations}\label{sec:val and trop}

In this section we merge the concepts of Khovanskii-finite valuations and the tropicalization of an ideal.
This section is based on results in \cite{KM16} and \cite{Bos20_fullrank}.

\subsection{Valuations from tropicalization}\label{sec:val from trop}
The main aim of Kaveh and Manon in \cite{KM16} is to establish a connection between the toric degenerations from prime cones in a tropicalization to toric degenerations obtained from Newton--Okounkov polytopes.
It relies on the quasivaluations with weighting matrices introduced above.

As before, let $I\subset S$ be a homogeneous ideal.
Suppose there exists a maximal prime cone $\tau\in\trop(I)$ and choose a basis $r_1,\dots,r_d\in \mathbb Q^n$ for the real vector space spanned by $\tau$. 
The quotient $\tau/\mathcal L_I$ is a strongly convex cone (see e.g. \cite[Lemma 2.13]{BMN}) so we may take a maximal linearly independent set of cosets of primitive ray generators of $\tau/\mathcal L_I$. 
Together with a basis of the lineality space this will be our choice for ${\bf r}:=\{r_1,\dots,r_d\}$.
In particular, we set $r_i=\ell_i$ for $1\le i\le m$, see Lemma~\ref{lem:grading and lineality}.
Define
\[
M_{\bf r}:=(r_{ij})_{1\le i\le d,1\le j\le n}
\]
where $r_{ij}$ is the $j^{\text{th}}$ entry in $r_i$, so the $r_i$ are the rows of $M_{\bf r}$.

\begin{Proposition}[Proposition 4.2 and 4.6 in \cite{KM16}]
If $\tau$ is a maximal prime cone in $\trop(I)$ then quasivaluation with weighting matrix $M_{\bf r}$ is in fact a full rank valuation with one-dimensional leaves.
Its value semigroup is generated by the images of $\bar x_1,\dots,\bar x_n$ and it only depends on $\tau$, not on our choice of ${\bf r}$.
\end{Proposition}

Given the Proposition we adopt our notation and set $M_{\tau}:=M_{\bf r}$ and $\nu_\tau:=\nu_{M_\tau}$.
We obtain the following corollary about the associated Newton--Okounkov polytopes:

\begin{wrapfigure}[14]{r}{2cm}
\centering
\begin{tikzpicture}[scale=.8]
    \node[teal] at (.75,5) {\small $S(A,\nu_\tau)$};
    \draw (-.5,0) -- (2.5,0);
    \draw (0,-.5) -- (0,6.5);
    \draw[thick, teal] (1,0) -- (1,3);
    \draw[magenta,thick] (0,0) -- (2.25,0);
    \draw[magenta,thick] (0,0) -- (2.25,6.75);
    \draw[fill,magenta,opacity=.2] (0,0) -- (2.25,0) -- (2.25,6.75) -- (0,0);
    \node[below] at (1,0) {\tiny 1};
    \node[below] at (2,0) {\tiny 2};
    \node[left] at (0,1) {\tiny 1};
    \node[left] at (0,2) {\tiny 2};
    \node[left] at (0,3) {\tiny 3};  
    \node[left] at (0,4) {\tiny 4};
    \node[left] at (0,5) {\tiny 5};
    \node[left] at (0,6) {\tiny 6};  
    \node[teal] at (0,0) {\footnotesize $\times$};
    \node[teal] at (1,0) {\footnotesize $\times$};
    \node[teal] at (1,2) {\footnotesize $\times$};
    \node[teal] at (1,3) {\footnotesize $\times$};    
    \node[magenta] at (1,1) {\tiny $\circ$};
    \node[teal] at (2,0) {\footnotesize $\times$};
    \node[magenta] at (2,1) {\tiny $\circ$};
    \node[teal] at (2,2) {\footnotesize $\times$};
    \node[teal] at (2,3) {\footnotesize $\times$};    
    \node[teal] at (2,4) {\footnotesize $\times$};
    \node[magenta] at (2,5) {\tiny $\circ$};
    \node[teal] at (2,6) {\footnotesize $\times$}; 
\end{tikzpicture}
\end{wrapfigure}

\begin{Corollary}[Corollary 4.7 in \cite{KM16}]
\label{cor:NO columns}
The Newton--Okounkov body of the valuation $\nu_\tau$ is a convex polytope whose vertices are $\nu_\tau(\bar x_1),\dots,\nu_\tau(\bar x_n)$, which are exactly the columns of $M_\tau$. 
Moreover, up to linear isomorphism $\Delta(A,\nu_\tau)$ only depends on $\tau$.
\end{Corollary}

Notice that for $I$ homogeneous with respect to the standard grading and $r_1=(1,\dots,1)$ as above the Corollary implies that (up to linear isomorphism) we have $\Delta(A,\nu)\subset \{1\}\times \mathbb R^{d-1}$.

\begin{Example}\label{example:val from trop}
Consider $I=(y^2z-x^3+z^3)\subset \mathbb C[x,y,z]$ from Example~\ref{example: curve} above. The initial ideal $(y^2z-x^3)$ corresponding to the ray $(2,3,0)^T$ in Figure~\ref{fig:example Gfan} is a maximal prime cone $\tau$ in $\trop(I)$. The associated ray matrix is $M_\tau=\left(\begin{smallmatrix}1&1&1\\2&3&0\end{smallmatrix}\right)$.
Let $A=\mathbb C[x,y,z]/I$, then the valuation $\nu_\tau:A\setminus\{0\}\to \mathbb Z_{\ge 0}\times\mathbb Z$ satisfies $\nu_\tau(\bar x)=(1,2),\nu_\tau(\bar y)=(1,3)$ and $\nu_\tau(\bar z)=(1,0)$.
In particular, the value semigroup of $\nu_\tau$ is generated by these three elements and depicted on the right.

Here $\textcolor{teal}{\times}$ denotes the lattice points in $S(A,\nu_\tau)$ and $\textcolor{magenta}{\circ}$ denotes lattice points not contained in $S(A,\nu_\tau)$.
The shaded region is the Newton--Okounkov cone $C(A,\nu_\tau)$ and the line segment connecting $(1,0)^T$ and $(1,3)^T$ is the Newton--Okounkov polytope $\Delta(A,\nu_\tau)$. 
Note that $S(A,\nu_\tau)$ is not saturated: for example, $(2,2)^T$ is in $S(A,\nu_\tau)$, but $(1,1)^T$ is not.
\end{Example}

\subsection{Tropicalization from valuations}\label{sec:trop from val}
Fix a Khovanskii-finite valuation $\nu:A\setminus\{0\}\to \mathbb Z^d$ on a positively graded algebra and domain $A$.
We may assume without loss of generality that $\dim_{\text{Krull}}(A)=d$ (if this was not the case we may apply \cite[Proposition 2.17(e)]{BrunsGubeladze} as the image of $\nu$ is in fact a monoid whose only unit is $0$).
Moreover, we may assume that the underlying total order on $\mathbb Z^d$ is the lexicographic order (if this was not the case we may follow Mora and Robbiano \cite{MoraRobbiano} and represent the order by a $d\times s$ matrix $M$ such that our order coincides with the lexicographic order on $M\mathbb Z^d$).
Choose a finite generating set $a_1,\dots,a_n$ for the value semigroup $\nu(A\setminus\{0\})$ and let $M_\nu$ be $d\times n$ matrix whose columns are $a_1,\dots,a_n$.
Notice that
\[
\text{rank}(M_\nu)=\dim(\text{im}(M_\nu)=\dim(\text{span}_{\mathbb Z}(a_1,\dots,a_n))=\dim(\text{cone}(a_1,\dots,a_n))=\text{rank}(\nu).
\]
In particular, $M_\nu$ is of full rank. 

\begin{Lemma}
Given the generators $a_1,\dots,a_n$ of the value semigroup $S(A,\nu)$ choose $b_1,\dots,b_n\in A$ with $\nu(b_i)=a_i$.
Then the set $\{b_1,\dots,b_n\}$ is a Khovanskii basis. 
\end{Lemma}
\begin{proof}
As $k[S(A,\nu)]\cong \gr_v(A)$ the elements $a_1,\dots,a_n$ form a set of algebra generators for $\gr_v(A)$. 
\end{proof}

Notice further that for dimension reasons the Khovsankii basis $\{b_1,\dots,b_n\}$ is a generating set for $A$ as $\nu$ is full-rank. Hence, we may use it to obtain a presentation of $A$. Define
\[
\pi:S:=k[x_1,\dots,x_n]\to A, \quad \text{by} \quad x_i\mapsto b_i.
\]
Notice that $b_1,\dots,b_n$ not necessarily are of degree one in $A$. 
To have a \emph{graded} presentation of $A$ we endow the polynomial ring with a (not necessarily standard) grading given by $\deg(x_i):=\deg(b_i)$.
Then $I:=\ker(\pi)$ is a homogeneous ideal and
$S/I\cong A$.
Our valuation $\nu$ is intimately related with the tropicalization of the ideal $I$. 

\begin{Theorem}[\cite{Bos20_fullrank}]\label{thm:val and trop}
Let $\nu:A\setminus \{0\}\to \mathbb Z^{d}$ be a full rank valuation with finitely generated value semigroup and let $S/I\cong A$ be the presentation induced by a Khovanskii basis $b_1,\dots,b_n$.

Then there exists a maximal prime cone $\tau\in \trop(I)$ such that $\nu=\nu_\tau$ and 
\[
k[S(A,\nu)] \cong \gr_\nu(A)\cong S/\init_\tau(I).
\]
\end{Theorem}

\begin{proof}
Notice that $M:=M_\nu$ the weighting matrix of $\nu$ is of full rank $d \le n$ as $\nu$ is of full rank.
Then by \cite[Theorem 1]{Bos20_fullrank} the initial ideal $\init_M(I)$ is prime as the value semigroup $S(A,\nu)$ is generated by $\nu(b_1),\dots,\nu(b_n)$. 
Here we use the total order on  $\mathbb Z^{d}$ given by $\nu$.
We may restrict our attention to the case of usual initial ideals as by \cite[Lemma 3]{Bos20_fullrank} there exists a weight vector $w\in \mathbb Z^{n}$ such that $\init_w(I)=\init_M(I)$.
It is left to show that
\begin{enumerate}
    \item $w \in \trop(J)$;
    \item $k[S(A,\nu)] \cong S/\init_\tau(J)$, where $w \in \tau^\circ$.  
\end{enumerate}
The first item follows from \cite[Corollary 3]{Bos20_fullrank}.
For the second consider the quasivaluation $\nu_M$.
By Proposition 2.1 $\nu_M$ is a valuation and by \cite[Proposition 1]{Bos20_fullrank} it satisfies $\nu=\nu_M$. 
Further, by \cite[Equation 3.3]{Bos20_fullrank} we have
\[
S/\init_w(I) \cong \gr_{\nu}(A) \cong k[S(A,\nu)],
\]
which finishes the proof.
\end{proof}

One nice feature of this connection is that it may be used to characterize when a toric degeneration is a Gr\"obner degeneration.
We explore this direction further in the following subsection.

Theorem~\ref{thm:val and trop} depends on a choice of Khovanskii basis, so naturally one may ask how strong this dependence is.
If we change the Khovanskii basis, the presentation of $A$ changes and so does the tropicalization.

\begin{Proposition}\label{prop:different trop}
Assume $\mathcal B:=\{b_1,\dots,b_n\}$ and $\mathcal B':=\{b'_1,\dots,b_n'\}$ are two Khovsankii bases of $(A,\nu)$.
Let $I$ and $I'$ be the ideals presenting $A$ and let $\tau$ and $\tau'$ be the cones in the corresponding tropicalizations from Theorem~\ref{thm:val and trop}.
Then there exists an ideal $\tilde I\subset k[y_1,\dots,y_m]$ for some $m\ge n$ presenting $A$ and projections $p:\mathbb R^m\to \mathbb R^n$ and $p':\mathbb R^m\to \mathbb R^n$ such that for a maximal prime cone $\tilde \tau\in \trop(\tilde I)$ we have $p(\tilde\tau)=\tau$ and $p'(\tilde\tau)=\tau'$.
\end{Proposition}

\begin{proof}
We have two presentations of $A$:
\[
\pi:S\to A,\pi(x_i)=b_i \quad \text{and} \quad \pi':S\to A,\pi'(x_i)=b_i'
\]
given by $\ker(\pi)=I$ and $\ker(\pi')=I'$. To see how the two tropicalizations $\trop(I)$ and $\trop(I')$ are related we proceed recursively and introduce another presentations of $A$ given by $\mathcal B\cup \mathcal B'$.
For simplicity assume $b_i=b_i'$ for all $i<n$.
Consider $\tilde \pi:k[x_1,\dots,x_{n+1}]\to A$ given by $x_i\mapsto b_i, x_{n+1}\mapsto b_n'$ and let $\tilde I:=\ker(\tilde \pi)$.
Let $p:k[x_1,\dots,x_{n+1}]\to k[x_1,\dots,x_n]$ and $p':k[x_1,\dots,x_{n-1},x_{n+1}]$ be the natural projections.
By construction we have $I\subset p(\tilde I)$ and $I'\subset p'(\tilde I)$.
Let $\tilde\tau\in\trop(\tilde I)$ be the maximal prime cone given by Theorem~\ref{thm:val and trop}.
Then the corresponding projections $p$ and $p'$ from $\mathbb R^{n+1}\to \mathbb R^n$ have the desired properties.
\end{proof}

Proposition~\ref{prop:different trop} invites us to change our point of view: suppose we have two Khovanskii-finite valuations $\nu$ and $\nu'$ on $A$ with two different Khovanskii bases $\mathcal B$ and $\mathcal B'$. 
We may use the proof of Proposition~\ref{prop:different trop} to construct a tropicalization that contains simultaneously prime cones corresponding to $\nu$ and $\nu'$.
This idea is closely related to a procedure in \cite{BLMM} for constructing new prime cones by changing the presentation of $A$ (\cite[Procedure 1]{BLMM}).  
It was shown in \cite{IW20} that for non Khovanskii-finite valuations the above mentioned procedure does not terminate.
We further elaborate on these ideas in the example of the Grassmannian $\Gr_3(\mathbb C^6)$ in \S\ref{sec:Gr36}.

\subsection{Which toric degenerations are Gr\"obner?}\label{sec:gröbner?}

Theorem~\ref{thm:val and trop} shows that toric degenerations induced by Khovanskii-finite valuations equivalently arise as Gröbner degenerations from the tropicalization of an adequate ideal. 
Naturally we may extend the question and ask which toric degenerations are Gröbner degenerations.

\medskip
Recall the definition of a toric degeneration from \S\ref{sec:intro}.
Assume $X$ is a projective variety and we have $X=\proj(A)$.
Given $\nu:A\setminus \{0\}\to \mathbb Z^d$ a Khovanskii-finite valuation, $\gr_\nu(A)$ is a $\mathbb Z^d$-graded toric algebra. 
Anderson showed how to construct a toric degenerations of $X$ given $\nu$ and we briefly recall his construction \cite{An13}.
In order to associate a (Noetherian) Rees algebra to $\nu$ that deforms $A$ to $\gr_\nu(A)$ we apply a standard trick  \cite[Proposition 1.8]{Bay82} to change from the $\mathbb Z^d$-grading on $\gr_\nu(A)$ to a $\mathbb Z$-grading:

\begin{Lemma}\label{lem:trick Bayer}
Let $F$ be a finite subset of $\mathbb Z^d$. Then there exists an {\bf order preserving projection} $e :\mathbb Z^d \to \mathbb Z_{\ge 0}$ such that for
all $m, n \in F$ we have $m<n$ (in $\mathbb Z^d$) implies  $e(m) < e(n)$ (in $\mathbb Z$).
\end{Lemma}

In our setting the set $F$ is induced by a Gröbner basis.
Consider a presentation of $A=S/I$ given by a Khovanskii-basis for $\nu$ and let $\tau\in \trop(I)$ be the maximal prime cone corresponding to $\nu$. Then choose a maximal cone $C\in \GF(I)$ who contains $\tau$ as a face and fix a Gröbner basis $g_1,\dots,g_s$ for $I$ with respect to the monomial initial ideal $\init_C(I)$.
We may assume without loss of generality that $\nu$ is the valuation associated to the matrix $M$ whose rows are representatives of the rays of $\tau$. 
If this is not the case by Corollary~\ref{cor:NO columns} there is a linear isomorphism that maps $S(A,\nu)$ to $S(A,\nu_M)$.
The valuation $\nu_M$ has the advantage that it compatible with $\tilde \nu_M:S\setminus \{0\}\to \mathbb Z^d$ defined as above by
\[
\tilde \nu_M(f) = \max_{\prec}\left\{Ma:f=\sum c_a{\bf x}^a, c_a\not=0\right\}
\]
where $\prec$ is the total order on $\mathbb Z^d$. For every element of the Gröbner basis $g$ we have an expression
\[
g=\sum_{a: \ Ma=\nu_M(g)} c_a{\bf x}^a + \sum_{b: \ Mb\prec \nu_M(g)} c_b{\bf x}^b,
\]
where in particular $Ma\succ Mb$. The elements $Ma,Mb$ for all $g$ in the Gröbner basis constitute the finite set $F$ of Lemma~\ref{lem:trick Bayer} that determines the order preserving projection $e:\mathbb Z^d \to \mathbb Z$.
It induces a $\mathbb Z$-filtration of $A$ with filtered pieces
\[
\mathcal F_{\nu,i}:=\{f\in A: e(\nu(f))\ge i\}.
\]
The associated graded algebra coincides with $\gr_\nu(A)$ by construction.
We define the {\bf Rees algebra of $\nu$} as
\[
\mathcal R_{\nu,A}:=\bigoplus_{i\ge 0} t^i\mathcal F_{\nu,i} \subset A[t].
\]
\begin{Proposition}[Proposition 5.1 in \cite{An13}]
The Rees algebra $\mathcal R_{\nu,A}$ is a flat $k[t]$-algebra with $\mathcal R_{\nu,A}/(t)\cong \gr_\nu(A)$ and $\mathcal R_{\nu,A}[t^{-1}]\cong A[t,t^{-1}]$ as $k[t]$-modules. 
\end{Proposition}

Moreover, the isomorphisms on the above proposition hold for the \emph{graded} algebras (remember that $\nu$ is homogeneous). 
In particular, $\phi:\proj(\mathcal R_{\nu,A})\to \mathbb A^1$ is a toric degeneration of $X$ with special fibre $\phi^{-1}(0)=X_0=\proj(\gr_\nu(A))$.
While $X_0$ is normal if and only if $S(A,\nu)$ is saturated, Anderson shows that its normalization is the projective toric variety associated with the Newton--Okounkov polytope $\Delta(A,\nu)$.

Given the example of a toric degeneration induced by a valuation we may formulate an \emph{algebraic} definition of toric degeneration. We summarize the definiton and its relation to valuations in the following result of Kaveh, Manon and Murata:

\begin{Theorem}[Theorem 1.11 in \cite{KMM_dgeneration-cx1}]
\label{thm:R and val}
Let $A$ be a positively graded domain and let $\mathcal R$ be a finitely generated positively graded $k[t]$-module and domain with the following properties:
\begin{itemize}
    \item $\mathcal R[t^{-1}]\cong A[t,t^{-1}]$ as $k[t]$-modules and graded algebras;
    \item the algebra $\mathcal R/(t)$ is a graded semigroup algebra $k[S]$ where $S\subset \mathbb Z_{\ge 0}\times \mathbb Z^d$;
    \item the standard $k^*$-action on $k[t]$ extends to an action on $\mathcal R$ respecting its grading, moreover this $k[t]$-action acts through $(k^*)^d$ on the semigroup algebra $k[S]$.
\end{itemize}
Then there is a full-rank valuation $\nu:A\setminus\{0\}\to \mathbb Z_{\ge 0}\times \mathbb Z^d$ such that $S=S(A,\nu)$.
\end{Theorem}

We call a toric degeneration of $\proj(A)$ induced by an algebra $\mathcal R$ as in Theorem~\ref{thm:R and val} an {\bf algebraic toric degeneration}. We obtain the following Corollary by combining Theorem~\ref{thm:R and val} and Theorem~\ref{thm:val and trop}:
 
\begin{Corollary}
Every algebraic toric degeneration is induced by a valuation and can be realized as a Gröbner degeneration associated with a maximal prime cone in the tropicalization of an apropriate ideal.
\end{Corollary}

\section{Adapted bases and wall-crossing formulas}\label{sec:adapted}

Suppose $A=\bigoplus_{j\ge 0}A_j$ is a graded algebra, for example the section ring of a line bundle, and it is equipped with a vector space basis $\BB$.
We assume that the basis is {\bf graded}, i.e. basis elements are homogeneous and a homogeneous element $f$ of degree $i$ is a linear combination of basis elements of degree $i$.
Recall that given a valuation $\nu:A\setminus\{0\}\to \Gamma$ the basis $\mathbb B$ is {\bf adapted to $\nu$} if for every $\gamma\in \Gamma$ the set $\mathbb B\cap \mathcal F_{\nu,\gamma}$ is a vector space basis of the filtered piece $\mathcal F_{\nu,\gamma}\subset A$.
Reversely we say that {\bf $\nu$ is adapted to $\mathbb B$}.
In particular, if $\nu$ has one-dimensional leaves we have a bijection of sets
\[
\mathbb B \quad \leftrightarrow \quad S(A,\nu).
\]
In this section we explore the consequences.
Assume the basis is {\bf parametrized by lattice points}, so we have an assignment of $b\mapsto m(b)\in \mathbb Z^n$ for all $b\in \mathbb B$.
In fact we may find an adapted basis with a parametrization by lattice points for every Khovanskii-finite valuation. 
Consider $\nu:A\setminus \{0\}\to \mathbb Z^{d}$ with Khovanskii basis $b_1,\dots,b_n$ and let $S/I\cong A$ and $\tau\in\trop(I)$ be the presentation of $A$ and the maximal prime cone in $\trop(I)$ as in Theorem~\ref{thm:val and trop}.
By definition $\trop(I)$ is a subfan of the Gröbner fan $\text{GF}(I)$ and $\tau$ is a face of at least one maximal cone in $\text{GF}(I)$.
Maximal cones in $\text{GF}(I)$ are in correspondence with monomial initial ideals of $I$ as defined above.
For a maximal cone $C$ we denote by $\init_C(I)\subset S$ the corresponding monomial ideal.
In particular, the set
\[
\mathbb B_C:=\{{\bf x}^m:{\bf x}^m\not \in \init_C(I)\}
\]
is a vector space basis for all quotients $A_w:=S/\init_w(I)$ with $w\in C$ called a {\bf standard monomial basis}.
Notice that for $w=0$ the quotient $A_w=A$ and for $w\in \tau^\circ$ we have $A_w\cong \gr_\nu(A)$.
Hence, $\mathbb B_C$ is an adapted basis for $\nu$.
The assignment
\[
{\bf x}^m\mapsto m\in \mathbb Z^n 
\]
is a parametrization by lattice points.
Recall that every monomial ideal has a unique set of monomial generators (see e.g. \cite[Proposition 1.1.6]{HH_monomial}).
Let ${\bf x}^{g_1},\dots,{\bf x}^{g_t}$ be this generating set.
Then ${\bf x}^m\in \init_C(I)$ if and only if there exists $i$ such that ${\bf x}^{g_i}\text{ divides } {\bf x}^m$.
This translates to $m\not \in \bigcup_{i=1}^t g_t+\mathbb Z^n_{\ge 0}$ for elements of the standard monomial basis, where $+$ denotes the Minkowski sum.
So we have bijections of sets
\[
S(A,\nu)  \quad \leftrightarrow \quad   \mathbb B_C \quad \leftrightarrow \quad  \mathbb Z^n_{\ge 0}\setminus \bigcup_{i=1}^t g_i+\mathbb Z^n_{\ge 0}
\]
Hence, we obtain the following corollary:

\begin{Corollary}
Every Khovanskii-finite valuation has an adapted basis parametrized by lattice points. 
\end{Corollary}

\begin{Example}\label{example: standard mono basis}
Consider as above $I=(y^2z-x^3+z^3)\subset \mathbb C[x,y,z]$ and the maximal prime cone $\tau\in \trop(I)$ spanned by the ray $(2,3,0)^T$ modulo $\mathcal L_I$.
Inside $\GF(I)$ the cone $\tau$ is adjacent to two maximal cones: one of them has associated initial monomial ideal $(x^3)$ (see Figure~\ref{fig:example Gfan}). Let $C$ be this maximal cone. Then $\mathbb B_C=\{x^ay^bz^c:a<3\}$ is the set of all monomials in $\mathbb C[x,y,z]$ that are not divisible by $x^3$, hence they are not in $\init_C(I)$.
\end{Example}

\subsection{Polytopes from adapted bases}\label{sec:polytopes from adapted}
Fix a Khovanskii-finite valuation $\nu:A\setminus\{0\}\to \mathbb Z^d$ and an associated adapted standard monomial basis $\mathbb B:=\mathbb B_C$ with the parametrization given above.
For an element $f\in A$ let $f=\sum_{{\bf x}^m\in \mathbb B}c_m{\bf x}^m$ be its linear extension in $\mathbb B$.
Let $M$ be the matrix whose rows $r_1,\dots,r_m$ are representatives of primitive ray generators of $C$. 
In particular, let $r_1,\dots,r_s$ be the generators of the lineality space.
Define $\supp_{\BB}(f):=\{m\in \mathbb Z^n:c_m\not =0\}\subset \mathbb Z^n$
and the {\bf $\BB$-Newton polytope}  of $f$ by
\[
\text{New}_{\BB}(f):=\conv\left(Ma:a\in\supp_{\BB}(f)\right)\subset \RR^m.
\]
Notice that $\text{New}_{\BB}(f)$ depends on our choice of ray generators for $C$. We will slightly abuse notation and not include $M$ in the index. 
We think of the $\text{New}_{\BB}(f)$ as placeholder for a valuation adapted to $\mathbb B$ and we define a placeholder for the Newton--Okounkov body of $\nu$:
\[
\Delta_{\BB}(A):=\conv\left(\bigcup_{j\ge 1}\left\{ \frac{1}{j}\text{New}_{\BB}(f):f\in A_j\right\}\right)\subset \mathbb R^m.
\]
Recall that the {\bf Newton--Okounkov body} of a valuation $\nu:A\setminus\{0\}\to \mathbb Z^d$ is defined as
\[
\Delta(A,\nu):=\conv\left(\bigcup_{j\ge 1}\left\{ \frac{\nu(f)}{j}:f\in A_j \right\}\right).
\]
If $\nu$ is fully homogeneous, i.e. of form $\nu:A\setminus \{0\}\to \mathbb Z_{\ge 0}^m\times \Gamma'$ given by $ \nu(f)=(\deg(f),\nu(f))$ then the above definition coincides with the one in Equation \eqref{eq:NObody}.
Any valuation constructed from a maximal prime cone $\tau$ in the tropicalization of an ideal as in \S\ref{sec:val from trop} is by construction fully homogeneous.
Recall that the first $m$ rows of $M_\tau$ are the elements $\ell_i$ from Lemma~\ref{lem:grading and lineality}.
Let $M_\ell$ be the submatrix with rows $\ell_1,\dots,\ell_m$.
Denote by ${\rm pr}:\mathbb Z_{\ge 0}^m\times \Gamma'\to \mathbb Z_{\ge 0}^m$ the projection.
Then for any element $f\in A$ we have
\[
{\rm pr}(\nu_M(f))=M_\ell b=\deg(f)
\]
where $b$ is such that $M_\tau b=\max_{<}\{M_\tau a:f=\sum c_a{\bf x}^a,c_a\not =0\}$.
Moreover, in this context the bijection between the basis $\mathbb B=\mathbb B_C$ and the valuation $\nu_{M_\tau}$ is given explicitly by
\begin{eqnarray*}
\mathbb B \to S(A,\nu_M),\quad  {\bf x}^a \mapsto M_{\tau}a
\end{eqnarray*}

\begin{Theorem}\label{thm:NObodies}
Let $C\in \GF(I)$ be a maximal cone that contains the maximal prime cones $\tau_1,\dots,\tau_q\in \trop(I)$ as $d$-dimensional face with associated Khovanskii-finite valuations $\nu_i:A\setminus\{0\}\to \mathbb Z^d$. 
Assume additionally that $\init_C(I)$ does not contain any variables.
Then there exist projections $p_i:\mathbb R^n\to \mathbb R^d$ for $1\le i\le q$ such that
\[
p_i(\Delta_{\mathbb B}(A))=\Delta(A,\nu_i).
\]
\end{Theorem}
\begin{proof}
Let $M$ be the matrix whose rows $r_1,\dots,r_{n_C}$ are either generators of the lineality space $\mathcal L_I$ or representatives of primitive ray generators for $C/\mathcal L_I$. 
Notice that $n_C\ge n$ as $C$ is a maximal cone with equality if and only if $C/\mathcal L_I$ is simplicial.
For every cone $\tau_i$ choose a collection of rows $r_1,\dots,r_s,r_{i_1},\dots,r_{i_{d-s}}$ 
(where $r_1,\dots,r_s$ are the generators of the lineality space), that correspond to rays spanning the same real vector space as $\tau_i$:
\[
\langle r_1,\dots,r_s,r_{i_1},\dots,r_{i_{d-s}}\rangle_{\mathbb R}= \langle \tau_i\rangle_{\mathbb R}.
\]
Denote the matrix whose rows are $r_1,\dots,r_s,r_{i_1},\dots,r_{i_{d-s}}$ by $M_i$
and define $p_i:\mathbb R^{n_C}\to \mathbb R^d$ as the projection onto the coordinates $1,\dots,s,i_1,\dots,i_{d-s}$.
Recall that $\Delta(A,\nu_i)$ without loss of generality by Corollary~\ref{cor:NO columns} is the convex hull of the columns of the matrix $M_i$.
We verify $p_i(\Delta_{\mathbb B}(A))=\Delta(A,\nu_i)$ pointwise by tracing the elements of $\mathbb B$ through both constructions.
As $\mathbb B$ is in bijection with $S(A,\nu_i)$ the claim follows.
Consider ${\bf x}^a\in \mathbb B$, then $\text{New}_{\mathbb B}({\bf x}^a)=Ma$. In $\Delta_{\mathbb B}(A)$ the element ${\bf x}^a$ corresponds to the point $\frac{1}{a_1+\dots+a_n}Ma$ and
\[
p_i\left(\frac{1}{a_1+\dots+a_n}Ma\right)= \frac{1}{a_1+\dots+a_n}M_ia=\frac{1}{a_1+\dots+a_n}\nu_i({\bf x}^a)\in \Delta(A,\nu_i),
\]
so $p_i(\Delta_{\mathbb B}(A))\subset \Delta(A,\nu_i)$. To show equality it suffices to verify that the vertices of $\Delta(A,\nu_i)$ are contained in $p_i(\Delta_{\mathbb B}(A))$.
By the additional assumption that $\init_C(I)$ does not contain any variables we know that $x_1,\dots,x_n\in \mathbb B$.
The computation above applied to a variable $x_j={\bf x}^a$ yields:
\[
p_i\left(\frac{1}{a_1+\dots+a_n}Ma\right)=M_ie_j=M_{ij},
\]
where $M_{ij}$ is the $j$th column of $M_i$ and a vertex of $\Delta(A,\nu_i)$ by Corolary~\ref{cor:NO columns}.
\end{proof}

\begin{Example}\label{exp:B-poly}
We continue with the Example~\ref{example: standard mono basis}. For the maximal cone $C$ we choose the ray matrix:
\[
M_C=\left(\begin{matrix} 1&1&1\\ 2&3&0\\ 1&0&1\end{matrix}\right).
\]
Notice that $(1,0,1)\mod \mathcal L_I=(0,-1,0)\mod \mathcal L_I$ which corresponds to the teal ray in Figure~\ref{fig:example Gfan}. Let $\mathbb B=\mathbb B_C$ and $A=\mathbb C[x,y,z]/I$.
Then 
\[
\Delta_{\mathbb B}(A)=\conv \left( \frac{1}{a+b+c}\left(\begin{smallmatrix}a+b+c\\2a+3b\\a+c\end{smallmatrix}\right):\begin{matrix}a+b+c\ge 1, a<3\\ a,b,c\in \mathbb Z_{\ge 0}\end{matrix}\right) =\conv\left(\left(\begin{smallmatrix}1\\2\\1\end{smallmatrix}\right),\left(\begin{smallmatrix}1\\3\\0\end{smallmatrix}\right),\left(\begin{smallmatrix}1\\0\\1\end{smallmatrix}\right)\right).
\]
Let $p_1$ be the projection away from the third coordinate in $\mathbb R^3$, then 
\[
p_1(\Delta_{\mathbb B}(A))=\conv\left(\binom{1}{3},\binom{1}{0}\right)=\Delta(A,\nu_\tau).
\]
Where $\tau\in \trop(I)$ is the maximal prime cone spanned by $(2,3,0)^T\mod \mathcal L_I$ as in Example~\ref{example:val from trop}.
\end{Example}

\subsection{Wall-crossing formulas}\label{sec:wall-crossing}
The Newton--Okounkov polytopes associated to the faces $\tau_1,\dots,\tau_q$ of $C$ are related by piecewise-linear maps called {\bf wall-crossing formulas} that were introduced by Escobar and Harada in \cite{EH-NObodies}.
We briefly review their construction.

Assume that $\tau_1$ and $\tau_2$ are two adjacent faces of the maximal cone $C\in \GF(I)$, so that $\tau:=\tau_1\cap\tau_2$ is a facet of both. 
Then we may choose the ray matrices $M_1$ and $M_2$ such that they agree in all rows but the last one.
Let $M_{1,2}$ be the matrix with the $d-1$ rows that $M_1$ and $M_2$ have in common.
In particular, we have two full-rank homogeneous valuations $\nu_1:=\nu_{\tau_1},\nu_2:=\nu_{\tau_2}: A\setminus \{0\}\to \mathbb Z^m_{\ge 0}\times \mathbb Z^{d-m}$ and one homogeneous valuation $\nu_{1,2}:A\setminus\{0\}\to \mathbb Z^m_{\ge 0}\times \mathbb Z^{d-m-1}$ of almost full rank, that is rank $d-1$.
We denote by $p_{[d-m-1]}:\mathbb R^{d-m}\to \mathbb R^{d-m-1}$ the projection onto the first $d-m-1$ coordinates.
By construction (and Corollary~\ref{cor:NO columns}) we have the following relation between the associated Newton--Okounkov polytopes
\[
\xymatrix{
\Delta(A,\nu_1) \ar[rd]_{p_{[d-m-1]}} & & \Delta(A,\nu_2) \ar[ld]^{p_{[d-m-1]}}\\
& \Delta(A,\nu_{1,2})
}
\]
In particular there exist piecewise linear maps $\varphi_i:\Delta(A,\nu_{1,2})\to \mathbb R$ and $\psi_i:\Delta(A,\nu_{1,2})\to \mathbb R$ for $i\in \{1,2\}$ such that
\begin{equation}
\Delta(A,\nu_i)=\left\{({\bf 1},v,z)\in \{{\bf 1}\}\times \mathbb R^{d-m-1}\times \mathbb R: \begin{matrix} ({\bf 1},v)\in \Delta(A,\nu_{1,2}) \\ \varphi_i({\bf 1},v)\le z\le \psi_i({\bf 1},v)\end{matrix}\right\},    
\end{equation}
where ${\bf 1}=(1,\dots,1)\in \mathbb R^m$.
By \cite[Theorem 3.4]{EH-NObodies} there exists a constant $\kappa>0$ such that for all $(1,v)\in \Delta(A,\nu_{1,2})$:
\[
\kappa(\psi_1({\bf 1},v)-\varphi_1({\bf 1},v)) = \psi_2({\bf 1},v)-\varphi_2({\bf 1},v)
\]
We define the piecewise linear wall-crossing maps
\begin{equation}\label{eq:flip and shift}
    \begin{split}
      S_{12}:\mathbb R^d \to \mathbb R^d & \text{ given by }
      ({\bf 1},v,z) \mapsto ({\bf 1},v,\kappa(z-\varphi_1({\bf 1},v))+\varphi_2({\bf 1,v})) \\
      F_{12}:  \mathbb R^d \to \mathbb R^d & \text{ given by } ({\bf 1},v,z) \mapsto ({\bf 1},v,\kappa(\varphi_1({\bf 1},v)-z)+\psi_2({\bf 1,v})).
    \end{split}
\end{equation}
The map $S_{12}$ is called the {\bf shift} and the map $F_{12}$ is called the {\bf flip}.

\begin{Theorem}[Theorem 2.7 in \cite{EH-NObodies}]
Let $I$ be a (multi-)homogeneous ideal in $S$ and $C$ a maximal cone in $\GF(I)$ such that there exist two maximal prime cones $\tau_1,\tau_2\subset C\cap\trop(I)$ that share a common facet $\tau=\tau_1\cap\tau_2$.
Let $\nu_1,\nu_2$ and $\nu_{1,2}$ be the associated homogeneous valuations. Then for $\Phi_{12}\in\{F_{12},S_{12}\}$ the associated Newton--Okounkov polytopes are related by
\[
\xymatrix{
\Delta(A,\nu_1) \ar[rd]_{p_{[d-m-1]}}\ar[rr]^{\Phi_{12}} & & \Delta(A,\nu_2) \ar[ld]^{p_{[d-m-1]}}\\
& \Delta(A,\nu_{1,2})
}
\]
and the Euclidean lengths of the fibers of $p_{[d-m-1]}$ are equal.
\end{Theorem}

\begin{Example}\label{exp:wall-crossing}
In our running example the maximal cone $C\in \GF(I)$ has two maximal prime cones in $\trop(I)$ as facets. Let $\tau_1\in \trop(I)$ be the cone generated by $(2,1,1)^T\mod \mathcal L_I$ (teal in Figure~\ref{fig:example Gfan}) and $\tau_2$ be the cone generated by $(2,3,0)^T\mod \mathcal L_I$.
Then
\[
\Delta(A,\nu_1)=\conv\left(\binom{1}{0},\binom{1}{3}\right) \quad \text{and} \quad \Delta(A,\nu_2)=\conv\left(\binom{1}{0},\binom{1}{1}\right).
\]
The Newton--Okounkov polytope $\Delta(A,\nu_{1,2})$ is simply the point $\{1\}\in \mathbb R$. Hence, the piecewise linear functions $\varphi_i,\psi_i$ are constants:
\[\begin{split}
\Delta(A,\nu_1)=\{(1,z):\varphi_1(1):=0\le z\le 3=:\psi_1(1)\},\\
\Delta(A,\nu_2)=\{(1,z): \varphi_2(1):=0\le z\le 1=:\psi_2(1)\}.
\end{split}
\]
The global constant can be computed from the volume of the Newton--Okounkov polytopes (with respect to the ambient subspace where they are full-dimensional poyltopes). We have that $\kappa_1{\rm vol}(\Delta(A,\nu_1))=\kappa_2{\rm vol}(\Delta(A,\nu_2))=\deg(y^2z-x^3+z^3)$ and $\kappa=|\kappa_1/\kappa_2|=\frac{1}{3}$, so
\[
S_{12}:(1,z)\mapsto \left(1,\frac{z}{3}\right), \quad F_{12}:(1,z)\mapsto \left(1,1-\frac{z}{3}\right).
\]
\end{Example}

\section{Families of Gröbner degenerations}\label{sec:families of gröbner}
In this section, we recall the main construction of the paper \cite{BMN}. It gives a multi-parameter flat family associated to a maximal cone $C\in \GF(I)$ where $I\subset S$ is a homogeneous ideal. 
We will see that this algebraic construction is closely related to the polyhedral objects from the previous section.
\medskip

Let $A$ be the quotient $S/I$. Recall the classical construction of a Gröbner degenerations associated to w a weight vector $w\in C^\circ$ from \S\ref{sec:initial ideals and trop} defined by the quotient $S[t]/I^{h;w}$.
The ideal $I^{h;w}$ defines the flat family ${\spec}(S[t]/I^{h;w}) \to {\spec}(k[t])$ whose fiber over the closed point $(t)$ is isomorphic to ${\spec}(A^w)$, where $A^w:=S/\init_w(I)$ and the fiber over any non-zero closed point $(t-c)$ is isomorphic to ${\spec}(A)$. 
Both, the construction of $I^{h;w}$ and \cite[Theorem 15.17]{Eisenbud} hold for arbitrary cones in $\GF(I)$.
In what follows, for simplicity, we focus on maximal cones as the generalization to lower dimensional ones is straight forward.

To generalize the construction of $I^{h;w}$ we fix vectors $r_1, \dots , r_{n_C} \in C$ such that $\{ \overline{r}_{1},\dots, \overline{r}_{n_C}\}$ is the set of primitive ray generators for $\overline{C}$, which is possible due to \cite[Lemma 2.13]{BMN}.
Let $M$ be the $(n_C \times n)$-matrix whose rows are $r_1,\dots, r_{n_C}$.  
Additionally, we write $<$ for a monomial term order compatible with $C$ and denote by $G$ the associated reduced Gr\"obner basis.

\begin{Definition}\thlabel{def:lift}
For $f=\sum_{\alpha \in \mathbb Z^n_{\ge 0}} c_\alpha {\bf x}^{\alpha} \in I$ set ${\mu}_M(f):=(\max_{c_\alpha\not =0}\{r_i\cdot \alpha\})_{i=1,\dots,n_C}\in \mathbb Z^{n_C\times 1}$, hence $\mu_M(f)$ as a column vector with $n_C$ entries.
Define the {\bf lift of $f$} as the polynomial $\tilde f_M \in S[t_1,\dots,t_{n_C}]$ given by the following formula 
\begin{displaymath}
    \tilde f_M := \tilde f_M ({\bf t},{\bf x}) 
    := f({\bf t}^{-M\cdot e_1}x_1,\dots, {\bf t}^{-M\cdot e_n}x_n) {\bf t}^{\mu_M(f)}
    =\sum_{\alpha \in \mathbb Z^n_{\ge 0}} c_\alpha {\bf x}^{\alpha} {\bf t}^{-M\cdot \alpha +\mu_M(f)}.
\end{displaymath}
Similarly, we define the {\bf lifted ideal} as $\tilde I_M:=\left\langle \tilde f_M : f \in I \right\rangle \subset S[t_1,\dots,t_{n_C}]$ and the {\bf lifted algebra} as the quotient 
\begin{equation}\label{eq:lifted algebra}
\tilde A_M:=S[t_1,\dots,t_{n_C}]/\tilde I_M.
\end{equation}
\end{Definition}
Although by construction the lifted algebra depends on the choice of ray matrix $M$ it can be shown that different choices yield the same algebra \cite[Corollary 3.10]{BMN}.
Another useful result about the lifted ideal $\tilde I_M$ is an explicit construction of a Gröbner basis. 
On $S[t_1,\dots,t_{n_C}]$ we consider the following term order induced by the term order $<$ on $S$ corresponding to $C$:
\begin{equation}\label{eq:<<}
{\bf x}^\alpha{\bf t}^{\lambda}\ll{\bf x}^\beta{\bf t}^{\mu}\quad\text{if and only if}\quad  \text{(i)}\ {\bf x}^\alpha<{\bf x}^\beta \quad \text{or}\quad \text{(ii)}\ {\bf x}^\alpha={\bf x}^\beta\ \text{and}\ {\bf t}^{\lambda}<_{\rm lex}{\bf t}^{\mu}.
\end{equation}
Then by \cite[Proposition 3.9]{BMN} the lifts of the elements of a Gröbner basis $G$ of $I$ with respect to $<$ form a Gröbner basis for $\tilde I_M$ with respect to $\ll$.
The main result is proven using this Gröbner basis and the standard monomial basis $\mathbb B_C$ of $A$:

\begin{Theorem}[Theorem 3.14 in \cite{BMN}]
\thlabel{thm:family}
Let $I$ be a homogeneous ideal in $S$, $A=S/I$, $C$ a maximal cone in $\GF(I)$ with associated term order $<$ and $M$ an $(m\times n)$-matrix whose rows are representatives of the primitive ray generators of $\overline{C}\subset \mathbb R^n/\mathcal L(I)$. Then:
\begin{itemize}
    \item[(i)] The algebra $\tilde A_M$ is a free $S$-module with basis $\mathbb B=\mathbb B_C$, the standard monomial basis of $A$ with respect to $\init_C(I)$. 
    In particular, we have a flat family 
    \begin{displaymath}
    \xymatrix{ {\rm Proj}(\tilde A_M) \ar@{^{(}->}[r]\ar[d]^\pi & \mathbb P^{n-1}\times \mathbb A^{n_C} \ar@{->>}[ld]\\
    \mathbb A^{n_C} &}
    \end{displaymath}
    \item[(ii)] For every face $\tau$ of $C$ there exists ${\bf a}_\tau\in \mathbb A^{n_C}$ such that
    $\pi^{-1}({\bf a}_\tau) = {\rm Proj}(S/\init_\tau(I))$. 
    In particular, generic fibers are isomorphic to ${\rm Proj}(A)$ and there exist special fibers for every proper face $\tau\subset C$.
\end{itemize}
\end{Theorem}

\begin{Example}\label{example:family}
In our running example $I=(f:=y^2z-x^3+z^3)$ the generator $f$ itself forms a Gröbner basis for every maximal cone in $\GF(I)$ (which is true more generally for hypersurfaces). 
Consider the maximal cone $C\in \GF(I)$ with its ray matrix $M_C$ from Example~\ref{example: standard mono basis}. 
For this construction we may omit the rays of $M_C$ coming from the lineality space. So that $ M=\left(\begin{smallmatrix}2&3&0\\1&0&1\end{smallmatrix}\right)$
Then $\mu_{M}(f)=\binom{6}{3}$ and 
\[
\tilde f_M=f(xt_0^{-1}t_1^{-2}t_2^{-1},yt_0^{-1}t_1^{-3},zt_0^{-1}t_2^{-1})t_0^3t_1^6t_2^3 = y^2zt_2^2-x^3+z^3t_1^6.
\]
Recall the two maximal prime cones $\tau_1,\tau_2\in \trop(I)$ from Example~\ref{exp:wall-crossing}. 
We have
\[
\tilde f_M\vert_{(t_1,t_2)=(0,1)} = y^2z-x^3=\init_{\tau_1}(f) \quad \text{and }\quad \tilde f_M\vert_{(t_1,t_2)=(1,0)} = -x^3+z^3=\init_{\tau_2}(f).
\]
\end{Example}

Notice that Theorem~\ref{thm:NObodies} may be interpreted as a polyhedral version of Theorem~\ref{thm:family}.
Let $C\in \GF(I)$ be the maximal cone and let $\tau_1,\tau_2$ be faces of $C$ that are maximal prime cones in $\trop(I)$.
Denote $X_1:=\proj(A_{\tau_1})$ and $X_2:=\proj(A_{\tau_2})$ and let $\nu_1$ and $\nu_2$ be the associated valuations. Recall that $\Delta(A,\nu_1)$ resp. $\Delta(A,\nu_2)$ is the polytope of the normalization of $X_1$ resp. $X_2$. We have the following diagrams
\[
\xymatrix{
X_1 \ar@{^{(}->}[r]\ar[d] & \proj(\tilde A_M)\ar[d]^\pi & X_2 \ar@{_{(}->}[l]\ar[d]\\
\{{\bf a}_{\tau_1}\}  \ar@{^{(}->}[r] & \mathbb A^{n_C} & \{{\bf a}_{\tau_2}\}  \ar@{_{(}->}[l]
} \quad 
\xymatrix{\Delta(A,\nu_1) & & \Delta(A,\nu_2)\\
& \Delta_{\mathbb B}(A)\subset \mathbb R^{n_C} \ar@{->>}[ul]^{p_1} \ar@{->>}[ur]_{p_1}
}
\]
If $\tau_1$ and $\tau_2$ are adjacent in the sense that $\tau:=\tau_1\cap \tau_2$ is a common facet of both we additionally have access to the wall-crossing formulas $F_{12}$ and $S_{12}$ defined in Equation \eqref{eq:flip and shift}. Let $\Phi_{12}\in \{F_{12},S_{12}\}$:
\begin{equation}\label{diagram}
\xymatrix{
& \Delta_{\mathbb B}(A) \ar@{->>}[dl]^{p_1} \ar@{->>}[dr]_{p_2}\\
\Delta(A,\nu_1)\ar[rr]_{\Phi_{12}}\ar[dr]_{p_{[d-m-1]}} & & \Delta(A,\nu_2)\ar[dl]^{p_{[d-m-1]}}\\
& \Delta(A,\nu_{1,2}) &
}
\end{equation}
The upper triangle of the diagram is not necessarily commutative as the following example shows: 

\begin{Example}\label{exp:algebraic wall-crossing}
Let $\tau_1,\tau_2$ and $C$ be as in the previous examples. We have previously computed the polytopes $\Delta_{\mathbb B}(A)=\conv\left(\left(\begin{smallmatrix}1\\2\\1\end{smallmatrix}\right),\left(\begin{smallmatrix}1\\3\\0\end{smallmatrix}\right),\left(\begin{smallmatrix}1\\0\\1\end{smallmatrix}\right)\right),\Delta(A,\nu_1)=\conv\left({ \binom{1}{2},\binom{1}{3},\binom{1}{0}}\right),$  and $ \Delta(A,\nu_2)=\conv\left(\binom{1}{1},\binom{1}{0},\binom{1}{1}\right)$,
where all the integral points are written in the order of the variables $x,y,z$ that induce them.
Recall the maps $S_{12}$ and $F_{12}$ from Example~\ref{exp:wall-crossing}. We have
\[\begin{split}
S_{12}: \binom{1}{2}\mapsto \binom{1}{2/3}, \ \binom{1}{3}\mapsto \binom{1}{1}, \ \binom{1}{0}\mapsto \binom{1}{0},\\
F_{12}: \binom{1}{2}\mapsto \binom{1}{1/3}, \ \binom{1}{3}\mapsto \binom{1}{0}, \ \binom{1}{0}\mapsto \binom{1}{1}.
\end{split}
\]
So in this case neither $F_{12}$ nor $S_{12}$ make the upper triangle in \eqref{diagram} commute.
\end{Example}

However, the bijections $\mathbb B\leftrightarrow S(A,\nu_i)$ for $i\in \{1,2\}$ induce a map  $A_{12}:S(A,\nu_1)\to S(A,\nu_2)$ that makes the upper triangle in \eqref{diagram} commutative. This map is called the {\bf algebraic wall-crossing} in \cite{EH-NObodies} and it is simply the composition of the bijections.
In special cases the algebraic wall-crossing coincides with the flip map. An example is given in \cite[\S5]{EH-NObodies} where the authors show this is the case for the Grassmannian of planes.
Moreover, in this case the flip map is the {\em Fock--Goncharov tropicalization} (see e.g. \cite[Remark 2.3]{GHK_birat}) of Fomin--Zelevinsky's {\em cluster mutation} (see e.g. \cite{FZ02}) as is shown in \cite[\S4.6]{BMN}.

\begin{Example}
The algebraic wall-crossing $A_{12}:S(A,\nu_1)\to S(A,\nu_2)$ is different from both of them. It is simply a bijection of sets induced by the bijections of $\mathbb B$ with $S(A,\nu_1)$ and $S(A,\nu_2)$. In particular, we have
\[
A_{12}: \binom{1}{2}\mapsto \binom{1}{1}, \ \binom{1}{3}\mapsto \binom{1}{0},\ \binom{1}{0}\mapsto \binom{1}{1}.
\]
\end{Example}

\section{Example: the Grassmannian \texorpdfstring{$\Gr_3(\mathbb C^6)$}{of 3-planes in six-spac ec}}\label{sec:Gr36}
In this section we apply the Proposition~\ref{prop:different trop} to the Grassmannian $\Gr_3(\mathbb C^6)$, or more precise to its homogeneous coordinate ring with respect to the Plücker embedding. 
The {\bf Plücker algebra} $A_{3,6}$ is quotient of the polynomial ring
\[
S:=\mathbb C[p_{ijk}:1\le i<j<k\le 6]
\]
by the {\bf Plücker ideal} $I_{3,6}$ that is generated minimally by the following 34 Plücker relations:
{\small
\begin{align*}
p_{256}p_{346} - p_{246}p_{356} + p_{236}p_{456} &&
p_{156}p_{346} - p_{146}p_{356} + p_{136}p_{456} &&
p_{256}p_{345} - p_{245}p_{356} + p_{235}p_{456} \\
p_{246}p_{345} - p_{245}p_{346} + p_{234}p_{456} &&
p_{236}p_{345} - p_{235}p_{346} + p_{234}p_{356} &&
p_{156}p_{345} - p_{145}p_{356} + p_{135}p_{456} \\
p_{146}p_{345} - p_{145}p_{346} + p_{134}p_{456} &&
p_{136}p_{345} - p_{135}p_{346} + p_{134}p_{356} &&
p_{156}p_{246} - p_{146}p_{256} + p_{126}p_{456} \\
p_{236}p_{245} - p_{235}p_{246} + p_{234}p_{256} &&
p_{156}p_{245} - p_{145}p_{256} + p_{125}p_{456} &&
p_{146}p_{245} - p_{145}p_{246} + p_{124}p_{456} \\
p_{126}p_{245} - p_{125}p_{246} + p_{124}p_{256} &&
p_{156}p_{236} - p_{136}p_{256} + p_{126}p_{356} &&
p_{146}p_{236} - p_{136}p_{246} + p_{126}p_{346} \\
p_{156}p_{235} - p_{135}p_{256} + p_{125}p_{356} &&
p_{145}p_{235} - p_{135}p_{245} + p_{125}p_{345} &&
p_{136}p_{235} - p_{135}p_{236} + p_{123}p_{356} \\
p_{126}p_{235} - p_{125}p_{236} + p_{123}p_{256} &&
p_{146}p_{234} - p_{134}p_{246} + p_{124}p_{346} &&
p_{145}p_{234} - p_{134}p_{245} + p_{124}p_{345} \\
p_{136}p_{234} - p_{134}p_{236} + p_{123}p_{346} &&
p_{135}p_{234} - p_{134}p_{235} + p_{123}p_{345} &&
p_{126}p_{234} - p_{124}p_{236} + p_{123}p_{246} \\
p_{125}p_{234} - p_{124}p_{235} + p_{123}p_{245} &&
p_{136}p_{145} - p_{135}p_{146} + p_{134}p_{156} &&
p_{126}p_{145} - p_{125}p_{146} + p_{124}p_{156} \\
p_{126}p_{135} - p_{125}p_{136} + p_{123}p_{156} &&
p_{126}p_{134} - p_{124}p_{136} + p_{123}p_{146} &&
p_{125}p_{134} - p_{124}p_{135} + p_{123}p_{145} 
\end{align*}
\begin{displaymath}
\begin{split}
p_{126}p_{345} - p_{125}p_{346} + p_{124}p_{356} - p_{123}p_{456}& \quad
p_{136}p_{245} - p_{135}p_{246} + p_{134}p_{256} + p_{123}p_{456}\\
p_{146}p_{235} - p_{135}p_{246} + p_{125}p_{346} + p_{123}p_{456}& \quad
p_{156}p_{234} - p_{134}p_{256} + p_{124}p_{356} - p_{123}p_{456}\\
p_{145}p_{236} - p_{135}p_{246}  + p_{125}p_{346} + p_{156}p_{234}
\end{split}
\end{displaymath}}

The tropicalization of $I_{3,6}$ was computed back in 2004 by Speyer and Sturmfels who found out that there are several maximal prime cones in $\trop(I_{3,6})$, \cite[\S5]{SS04}. We briefly summarize their findings.

\begin{Theorem}
The tropical Grassmannian $\trop(I_{3,6})\subset \mathbb R^{20}/\mathcal L_{I_{3,6}}$ is a four-dimensional fan with 1005 maximal cones, 990 of which are prime.
The 65 rays consist of 
\begin{enumerate}
    \item 20 standard basis elements $e_{ijk}$, $1\le i<j<k\le 6$;
    \item 15 vectors of form $f_{ij}=\sum_{k\not\in \{i,j\}}e_{ijk}$, $1\le i<j\le 6$;
    \item 30 vectors associated with 15 tripartitions $\{\{i_1,i_2\},\{i_3,i_4\,\{i_5,i_6\}\}$ of $[6]$ each defining two rays
    $g_{i_1i_2i_3i_4i_5i_6}:= f_{i_5i_6}+e_{i_3i_4i_5}+e_{i_3i_4i_6}$  \text{and}  $g_{i_1i_2i_5i_6i_3i_4}:= f_{i_3i_4}+e_{i_3i_5i_6}+e_{i_4i_5i_6}$. 
\end{enumerate}
The symmetric group action on the index sets of Plücker coordinates induces a symmetry on $\trop(I_{3,6})$. The maximal cones are grouped in seven orbits, six of which consist of prime cones:
\begin{itemize}
    \item[EEEE] there are 30 simplicial prime cones of type $\{e_{123}, e_{145}, e_{246}, e_{356}\}$;
    \item[EEFF1] there are 90 simplicial prime cones of type $\{e_{123},e_{456}, f_{56}, f_{12}\}$;
    \item[EEFF2] there are 90 simplicial prime cones of type $\{e_{125},e_{345}, f_{12}, f_{34}\}$;
    \item[EFFG] there are 180 simplicial prime cones of type $\{e_{345}, f_{34}, f_{12},g_{123456}\}$;
    \item[EEEG] there are 240 simplicial prime cones of type $\{e_{126},e_{134},e_{356},g_{125634}\}$;
    \item[EEFG] there are 360 simplicial prime cones of type $\{e_{234}, e_{125}, f_{34}, g_{125634}\}$;
    \item[FFFGG] there are 15 non-simplicial non-prime cones of type $\{f_{56}, f_{34}, f_{12},g_{123456},g_{125634}\}$.
\end{itemize}
\end{Theorem}

Many families of (full rank homogeneous) valuations are known for $A_{3,6}$ (see e.g. \cite{B-birat,FFL_birat,MoSh} ) and whenever the Plücker coordinates form a Khovanskii basis there is a unique maximal prime cone in $\trop(I_{3,6})$ associated with it by Theorem~\ref{thm:val and trop}.
It is however not true that all \emph{known} valuations on $A_{3,6}$ share the Plücker coordinates as a Khovanskii basis.
For an example, you may want to consider \cite[\S9]{RW17} where Rietsch and Williams exihibit an example of a valuation induced by a \emph{plabic graph}. 
This example can be generalized to higher Grassmannians, see \cite[\S5]{Bos20_fullrank}. 
For more details, including background on how to obtain valuations from plabic graphs we refer the reader to the mentioned references.

\begin{Example}\label{exp:bad plabic val}
Let $\nu_G:A_{3,6}\to \mathbb Z^9$ be the Rietsch--Williams valuation associated with the plabic graph on the left in Table~\ref{tab:3-6}.
The values of Plücker coordinates under $\nu_G$ can be found in Table~\ref{tab:3-6}.
The valuation is Khovanskii-finite, but the Plücker coordinates do not form a Khovanskii basis. Among the vertices of $\Delta(A_{3,6},\nu_G)$ there is one non integral of form
\[
\left(\frac{1}{2},\frac{1}{2},\frac{3}{2},\frac{1}{2},\frac{1}{2},1,\frac{3}{2},1,\frac{1}{2}\right).
\]
It is obtained from the ray of the Newton--Okounkov cone generated by $\nu_G(\bar p_{124}\bar p_{356}-\bar p_{123}\bar p_{456})$.
\end{Example}

The aim of this section is to illustrate how Proposition~\ref{prop:different trop} can be applied to find an appropriate tropicalization where the Khovanskii-finite valuation $\nu_G$ appears as associated with a maximal prime cone.

\begin{table}
    \centering
\begin{tabular}{ll}    
\begin{tikzpicture}
[scale=.4,decoration={markings, 
    mark= at position 0.5 with {\arrow{stealth}}}
] 
\draw (0,0) circle [radius=5];  
\draw[postaction={decorate}] (1,2) -- (-1,2);
\draw[postaction={decorate}] (1,2) -- (2.25,0);
\draw[postaction={decorate}] (2.25,0) -- (1,-2);
\draw[postaction={decorate}] (1,-2)-- (-1,-2);
\draw[postaction={decorate}] (-2.25,0) -- (-1,-2);
\draw[postaction={decorate}] (-1,2) -- (-2.25,0);
\draw[postaction={decorate}] (-1,3.5) -- (-1,2);
\draw[postaction={decorate}] (-1,3.5)-- (1,3.5);
\draw[postaction={decorate}] (1,3.5) -- (1,2);
\draw[postaction={decorate}] (3.5,-1) -- (2.25,0);
\draw[postaction={decorate}] (3.5,-1) -- (2.375,-3);
\draw[postaction={decorate}] (1,-2) -- (2.375,-3);
\draw[postaction={decorate}] (-2.25,0) -- (-3.5,-1);
\draw[postaction={decorate}] (-2.375,-3) -- (-3.5,-1);
\draw[postaction={decorate}] (-1,-2) -- (-2.375,-3);
\draw[postaction={decorate}] (-1,4.9) -- (-1,3.5);
\draw[postaction={decorate}] (1,4.9) -- (1,3.5);
\draw[postaction={decorate}] (4.7,-1.75) -- (3.5,-1);
\draw[postaction={decorate}] (2.375,-3) -- (3.35,-3.75);
\draw[postaction={decorate}] (-3.5,-1) -- (-4.7,-1.75);
\draw[postaction={decorate}] (-2.375,-3) -- (-3.35,-3.75);

\draw[fill] (1,3.5) circle [radius=.175];  
\draw[fill] (-1,2) circle [radius=.175];  
\draw[fill] (2.25,0) circle [radius=.175];  
\draw[fill] (2.375,-3) circle [radius=.175];  
\draw[fill] (-1,-2) circle [radius=.175];  
\draw[fill] (-3.5,-1) circle [radius=.175];  

\draw[fill, white] (-1,3.5) circle [radius=.175];  
\draw (-1,3.5) circle [radius=.175];  
\draw[fill, white] (1,2) circle [radius=.175];  
\draw (1,2) circle [radius=.175];  
\draw[fill, white] (3.5,-1) circle [radius=.175];  
\draw (3.5,-1) circle [radius=.175];
\draw[fill, white] (1,-2) circle [radius=.175];  
\draw (1,-2) circle [radius=.175];  
\draw[fill, white] (-2.25,0) circle [radius=.175];  
\draw (-2.25,0) circle [radius=.175];  
\draw[fill, white] (-2.375,-3) circle [radius=.175];  
\draw (-2.375,-3) circle [radius=.175];  

\node[above] at (-1,4.9) {1};
\node[above] at (1,4.9) {2};
\node[right] at (4.7,-1.75) {3};
\node[right] at (3.35,-3.75) {4};
\node[left] at (-4.7,-1.75) {6};
\node[left] at (-3.35,-3.75) {5};

\node at (0,0) {\scalebox{.3}{ $\yng(2,1)$}};
\node at (0,2.75) {\scalebox{.3}{$\yng(2)$}};
\node at (0,4.25) {\scalebox{.3}{$\yng(3)$}};
\node at (3,1.75) {\scalebox{.3}{$\yng(3,3)$}};
\node at (2.3,-1.5) {\scalebox{.3}{$\yng(3,3,2)$}};
\node at (3.7,-2.25) {\scalebox{.3}{$\yng(3,3,3)$}};
\node at (0,-3.5) {\scalebox{.3}{$\yng(2,2,2)$}};
\node at (-2.3,-1.5) {\scalebox{.3}{$\yng(1,1)$}};
\node at (-3.7,-2.25) {\scalebox{.3}{$\yng(1,1,1)$}};
\node at (-3,1.75) {{$\varnothing$}};

\end{tikzpicture}
&
    \begin{tabular}{c|c c c c c c c c c c}
    & \scalebox{.2}{$\yng(3)$} & \scalebox{.2}{$\yng(3,3)$} & \scalebox{.2}{$\yng(3,3,3)$} & \scalebox{.2}{$\yng(2,2,2)$} & \scalebox{.2}{$\yng(1,1,1)$} &  & \scalebox{.2}{$\yng(2)$} & \scalebox{.2}{$\yng(2,1)$} & \scalebox{.2}{$\yng(3,3,2)$} & \scalebox{.2}{$\yng(1,1)$} \\ \hline
    $\bar p_{123} $  & 0 & 0 & 0 & 0 & 0 & & 0 & 0 & 0 & 0 \\
    $\bar p_{124} $  & 0 & 0 & 1 & 0 & 0 & & 0 & 0 & 0 & 0 \\
    $\bar p_{125} $  & 0 & 0 & 1 & 1 & 0 & & 0 & 0 & 1 & 0 \\
    $\bar p_{126} $  & 0 & 0 & 1 & 1 & 1 & & 0 & 0 & 1 & 0 \\
    $\bar p_{134} $  & 0 & 1 & 1 & 0 & 0 & & 0 & 0 & 1 & 0 \\
    $\bar p_{135} $  & 0 & 1 & 1 & 1 & 0 & & 0 & 0 & 1 & 0 \\
    $\bar p_{136} $  & 0 & 1 & 1 & 1 & 1 & & 0 & 0 & 1 & 0 \\
    $\bar p_{145} $  & 0 & 1 & 2 & 1 & 0 & & 0 & 0 & 1 & 0 \\
    $\bar p_{146} $  & 0 & 1 & 2 & 1 & 1 & & 0 & 0 & 1 & 0 \\
    $\bar p_{156} $  & 0 & 1 & 2 & 2 & 0 & & 0 & 1 & 2 & 1 \\
    $\bar p_{234} $  & 1 & 1 & 1 & 0 & 0 & & 0 & 0 & 1 & 0 \\
    $\bar p_{235} $  & 1 & 1 & 1 & 1 & 0 & & 0 & 0 & 1 & 0 \\
    $\bar p_{236} $  & 1 & 1 & 1 & 1 & 1 & & 0 & 0 & 1 & 0 \\
    $\bar p_{245} $  & 1 & 1 & 2 & 1 & 0 & & 0 & 0 & 1 & 0 \\
    $\bar p_{246} $  & 1 & 1 & 2 & 1 & 1 & & 0 & 0 & 1 & 0 \\
    $\bar p_{256} $  & 1 & 1 & 2 & 2 & 1 & & 0 & 1 & 2 & 1 \\
    $\bar p_{345} $  & 1 & 2 & 2 & 1 & 0 & & 1 & 1 & 2 & 0 \\
    $\bar p_{346} $  & 1 & 2 & 2 & 1 & 1 & & 1 & 1 & 2 & 0 \\
    $\bar p_{356} $  & 1 & 2 & 2 & 2 & 1 & & 1 & 1 & 2 & 1 \\
    $\bar p_{456} $  & 1 & 2 & 3 & 2 & 1 & & 1 & 1 & 2 & 1 
    \end{tabular} 
\end{tabular}    
    \caption{the plabic graph $G$ for $\Gr_3(\mathbb C^6)$ for which $\Delta(A_{3,6},\nu_G)$ is not integral (see \cite[\S9]{RW17}) and the images of Pl\"ucker coordinates under the valuation $\val_{G}$ as in Example~\ref{exp:bad plabic val}.}
    \label{tab:3-6}
\end{table}

\subsection{Cluster embedding}
The Grassmannian $\Gr_3(\mathbb C^6)$ has a cluster structure that was first exhibited by Scott in \cite{Sco06}.
More precisely, the algebra $A_{3,6}$ is a cluster algebra which roughly means that is can be constructed recursively from particular sets of maximal algebraically independent elements, called \emph{seed} by a combinatorial procedure called \emph{mutation} \cite{FZ02}. 
The elements of the seeds are called the \emph{cluster variables}.
There are 22 of them for $A_{3,6}$, 20 of which are Plücker coordinates together with two more generators of degree two denoted by $X$ and $Y$ that are binomials in Plücker coordinates.
In particular, $Y$ agrees with the element $\bar p_{124}\bar p_{356}-\bar p_{123}\bar p_{456}$ from Example~\ref{exp:bad plabic val}.

We change the representation of $A_{3,6}$ to be $\mathbb C[p_{ijk},X,Y:1\le i<j<k\le 6]/J_{3,6}$, where the ideal $J_{3,6}$ is minimally generated by the following 37 quadratic polynomials:
{\small
\begin{align*}
p_{145}p_{236} - {p_{123}}{p_{456}} - X,\quad\quad \quad \quad & \quad
p_{124}p_{356} - {p_{123}}{p_{456}} - Y,\quad\quad \quad\quad & 
p_{136}p_{245} - {p_{126}}{p_{345}} - X,\quad \quad\quad \\
p_{125}p_{346} - {p_{126}}{p_{345}} - Y,\quad \quad\quad \quad & \quad
p_{146}p_{235} - {p_{156}}{p_{234}} - X,\quad \quad\quad \quad & 
p_{134}p_{256} - {p_{156}}{p_{234}} - Y,\quad\quad \quad \\
p_{246}p_{356} - p_{346}p_{256} - p_{236}{p_{456}}, \quad& \quad
p_{245}p_{356} - {p_{345}}p_{256} - p_{235}{p_{456}},\quad&
p_{146}p_{356} - p_{346}{p_{156}} - p_{136}{p_{456}}, \\ 
p_{145}p_{356} - {p_{345}}{p_{156}} - p_{135}{p_{456}},\quad& \quad
p_{245}p_{346} - {p_{345}}p_{246} - {p_{234}}{p_{456}}, \quad& 
p_{235}p_{346} - {p_{345}}p_{236} - {p_{234}}p_{356},\\ 
p_{145}p_{346} - {p_{345}}p_{146} - p_{134}{p_{456}}, \quad& \quad
p_{135}p_{346} - {p_{345}}p_{136} - p_{134}p_{356},\quad &
p_{146}p_{256} - p_{246}{p_{156}} - {p_{126}}{p_{456}}, \\
p_{145}p_{256} - p_{245}{p_{156}} - p_{125}{p_{456}},\quad&\quad
p_{136}p_{256} - p_{236}{p_{156}} - {p_{126}}p_{356}, \quad& 
p_{135}p_{256} - p_{235}{p_{156}} - p_{125}p_{356},\\ 
p_{235}p_{246} - p_{245}p_{236} - {p_{234}}p_{256}, \quad& \quad
p_{145}p_{246} - p_{245}p_{146} - p_{124}{p_{456}},\quad&
p_{136}p_{246} - p_{236}p_{146} - {p_{126}}p_{346}, \\
p_{134}p_{246} - {p_{234}}p_{146} - p_{124}p_{346},\quad&\quad
p_{125}p_{246} - p_{245}{p_{126}} - p_{124}p_{256}, \quad& 
p_{134}p_{245} - {p_{234}}p_{145} - p_{124}{p_{345}},\\ 
p_{135}p_{245} - p_{235}p_{145} - p_{125}{p_{345}}, \quad& \quad
p_{135}p_{236} - p_{235}p_{136} - {p_{123}}p_{356},\quad&
p_{134}p_{236} - {p_{234}}p_{136} - {p_{123}}p_{346}, \\
p_{125}p_{236} - p_{235}{p_{126}} - {p_{123}}p_{256},\quad& \quad
p_{124}p_{236} - {p_{234}}{p_{126}} - {p_{123}}p_{246}, \quad& 
p_{134}p_{235} - {p_{234}}p_{135} - {p_{123}}{p_{345}},\\
p_{124}p_{235} - {p_{234}}p_{125} - {p_{123}}p_{245}, \quad& \quad
p_{135}p_{146} - p_{145}p_{136} - p_{134}{p_{156}},\quad&
p_{125}p_{146} - p_{145}{p_{126}} - p_{124}{p_{156}}, \\
p_{125}p_{136} - p_{135}{p_{126}} - {p_{123}}{p_{156}},\quad& \quad
p_{124}p_{136} - p_{134}{p_{126}} - {p_{123}}p_{146}, \quad& 
p_{124}p_{135} - p_{134}p_{125} - {p_{123}}p_{145},
\end{align*}
\begin{displaymath}
p_{135}p_{246} - {p_{156}}{p_{234}} - Y - {p_{123}}{p_{456}} - X  - {p_{126}}{p_{345}}.
\end{displaymath}}

The tropicalization of the ideal $J_{3,6}$ is not known completely, but the intersection of the tropicalization with a specific maximal cone in its Gröbner fan was computed in \cite[\S4.4]{BMN}. We summarize their findings.

Let $\{e_{123},\dots,e_{456},e_x,e_y\}$ denote the standard basis of $\mathbb R^{22}$ we define (by slight abuse of notation) the elements $f_{ij}$ and $g_{abcdef}$ in $\mathbb R^{22}$ as the same linear combinations of standard basis elements as in $\mathbb R^20$. 
Then the lineality space $\mathcal L_{J_{3,6}}$ is six-dimensional and spanned by
\[
E_i:=\sum_{k,j\not=i}e_{ijk}+e_x+e_y.
\]
The ideal $J_{3,6}$ is invariant under the action of the group $\Sigma:=\langle(123456),(16)(25)(34)\rangle\subset S_6$.
In particular, this action translates to an action on $\trop(J_{3,6})$.

\begin{Theorem}[\S4.4 in \cite{BMN}]
There is a distinguished maximal simplicial cone $C\in \GF(J_{3,6})$ that is invariant under the action of $\Sigma$ with ray generators 
\begin{enumerate}
    \item 6 of form $e_{i,i+1,i+2}$ for $i\in \mathbb Z_6$;
    \item 6 of form $f_{i,i+1}+\left\{\begin{matrix}e_y& i \text{ odd}\\e_x& i \text{ evem}\end{matrix}\right.$;
    \item 2 of form $g_{123456}+e_y$ and $g_{456123}+e_x$;
    \item 2 of form $g_{654321}+e_y$ and $g_{321654}+e_x$.
\end{enumerate}
All $16$ rays are also rays of $\trop(J_{3,6})$ (more precisely of its totally positive part).
The intersection $C\cap \trop(J_{3,6})$ contains 50 maximal simplicial cones of $\trop(J_{3,6})$ that are all prime. Specifically we find
\begin{enumerate}
    \item 6 in the $\Sigma$-orbits of $\{e_{123},e_{156},f_{23}+e_x,f_{56}+e_y\}$ projecting onto cones of type EEFF1;
    \item 12 in the $\Sigma$-orbits of $\{e_{123},e_{456},f_{12}+e_y,f_{56}+e_y\}$ and $\{f_{34}+e_y,f_{16}+e_x,e_{156},e_{234}\}$ projecting onto cones of type EEFF2;
    \item 12 in the $\Sigma$-orbits of $\{e_{123},g_{321654}+e_x,f_{23}+e_x,f_{45}+e_x\}$ and $\{f_{34}+e_y,g_{123456}+e_y,f_{12}+e_y,e_{345}\}$ projecting onto cones of type EFFG;
    \item 4 in the $\Sigma$-orbit of $\{e_{123},g_{123456}+e_y,e_{156},e_{345}\}$ projecting onto cones of type EEEG;
    \item 12 in the $\Sigma$-orbits of $\{g_{456123}+e_x,e_{456},f_{23}+e_x,e_{234}\}$ and $\{g_{654321}+e_y,f_{56}+e_y,e_{456},e_{234}\}$ projecting onto cones of type EEFG;
    \item 4 in the $\Sigma$-orbit of $\{f_{45}+e_x,f_{16}+e_x,f_{23}+e_x,g_{321654}+e_x\}$ projecting onto FFFG type pyramids inside the maximal cones of bipyramid type FFFGG.
\end{enumerate}
Here the projections refer to the projection $\mathbb R^{22}\supset\trop(J_{3,6})\to \trop(I_{3,6})\subset\mathbb R^{20}$.
\end{Theorem}

The ideal $J_{3,6}$ has the following advantage over the ideal $I_{3,6}$:

\begin{Corollary}
The tropicalization $\trop(J_{3,6})$ contains maximal prime cones associated with all Khovanskii-finite valuations on  $A_{3,6}$ for which  $\{p_{123},\dots,p_{456},X,Y\}$ is a Khovanskii basis. In particular, this includes all valuations associated with plabic graphs for $\Gr_3(\mathbb C^6)$.
\end{Corollary}

In particular, in $\trop(J_{3,6})$ we find a maximal prime cone associated with the valuation from Example~\ref{exp:bad plabic val}. It is identified with the cone whose rays are $\{f_{45}+e_x,f_{16}+e_x,f_{23}+e_x,g_{456123}+e_x\}$.

\footnotesize{
\newcommand{\etalchar}[1]{$^{#1}$}

}

\end{document}